\documentclass[a4paper,12pt,reqno]{amsart}
\usepackage{setspace} 
\usepackage{amsmath, amsfonts,amsthm,amssymb,mathrsfs,slashed,cases}
\usepackage{enumerate}
\usepackage[dvipdfmx]{graphicx}
\usepackage{ascmac}
\usepackage[arrow,matrix]{xy}
\usepackage{color}
\usepackage{array}
\usepackage{pifont}
\usepackage{longtable}
\usepackage{mathtools}

\theoremstyle{definition}
\newtheorem{theorem}{Theorem}[section]
\newtheorem*{theorem*}{Theorem}

\newtheorem*{definition*}{Definition}
\newtheorem{lemma}[theorem]{Lemma}
\newtheorem*{lemma*}{Lemma}

\newtheorem*{example*}{Example}
\newtheorem{proposition}[theorem]{Proposition}
\newtheorem*{proposition*}{Proposition}
\newtheorem{corollary}[theorem]{Corollary}
\newtheorem*{corollary*}{Corollary}

\newtheorem{maintheorem}{Theorem}

\newcommand{\ctext}[1]{\raise0.2ex\hbox{\textcircled{\scriptsize{#1}}}}

\DeclareMathOperator{\Spin}{Spin}
\DeclareMathOperator{\pr}{pr}
\DeclareMathOperator{\Id}{Id}
\DeclareMathOperator{\tr}{tr}
\DeclareMathOperator{\Ric}{Ric}
\DeclareMathOperator{\scal}{scal}
\DeclareMathOperator{\vol}{vol}

\DeclareMathOperator{\SO}{SO}
\DeclareMathOperator{\GL}{GL}
\DeclareMathOperator{\G}{G}

\DeclareMathOperator{\SU}{SU}

\DeclareMathOperator{\U}{U}

\DeclareMathOperator{\Sp}{Sp}
\DeclareMathOperator{\End}{End}

\DeclareMathOperator{\Hom}{Hom}
\DeclareMathOperator{\Sym}{Sym}

\DeclareMathOperator{\diag}{diag}
\DeclareMathOperator{\Cas}{Cas}

\def\C{\mathbb{C}}

\title{Infinitesimal deformations of Killing spinors on nearly parallel $\G_2$-manifolds}

\author{Soma Ohno}

\address{Soma Ohno, Department of Pure and applied Mathematics, Graduate school of fundamental science
and engineering, Waseda University, 3-4-1 Ohkubo, Shinjuku-ku, Tokyo, 169-8555, Japan.}
\email{runhorse@fuji.waseda.jp}

\begin{document}

\begin{abstract}
Manifolds admitting Killing spinors are Einstein manifolds. Thus, a deformation of a Killing spinor entails a deformation of Einstein metrics. In this paper, we study infinitesimal deformations of Killing spinors on nearly parallel $\G_2$-manifolds. Since there is a one-to-one correspondence between nearly parallel $\G_2$-structures and Killing spinors on 7-dimensional spin manifolds, our results imply that infinitesimal deformations of nearly parallel $\G_2$-structures are examined in terms of Killing spinors. Applying the same technique, we identify that the space of the Rarita-Schwinger fields coincides with a subspace of the eigenspace of the Laplacian.
\end{abstract}

\maketitle

\makeatletter
  \renewcommand{\theequation}{
  \thesection.\arabic{equation}}
 \@addtoreset{equation}{section}
\makeatother

\section{Introduction}
Let $(M,g)$ be an $n$-dimensional Riemannian spin manifold and $S_{1/2}$ a spinor bundle on $M$. A non-trivial spinor field $\kappa \in \Gamma(S_{1/2})$ is called a Killing spinor if for some constant $c \in \mathbb{C}$ and all vector fields $X \in \mathfrak{X}(M)$ it satisfies
\[
\nabla_X \kappa = cX \cdot \kappa,
\]
where $\nabla$ is the spinor connection induced by the Levi-Civita connection and $\cdot$ denotes Clifford multiplication. This constant $c$ is usually called the Killing number. A simple calculation shows that the Ricci curvature satisfies $\Ric = 4c^2(n-1)g$, that is, manifolds with Killing spinors must be Einstein manifolds with the Einstein constant $E=4c^2(n-1)$. Thus, the Killing number $c$ is either zero, pure imaginary, or real. In the $c=0$ case, the corresponding Killing spinor is a parallel spinor. When $c$ is pure imaginary or real, it is called an imaginary or real Killing spinor, respectively. If $c$ is real and $M$ is complete, then it is compact by Myer's theorem. On the other hand, if $c$ is pure imaginary, $M$ is non-compact. 

The problem of classifying manifolds with these Killing spinors was well thought out. Here are some of the major results. The problem was solved by Wang \cite{Wang1}, \cite{Wang2} for manifolds with parallel spinors, by Baum \cite{Baum} for those with imaginary Killing spinors, and by B\"{a}r \cite{Bar} for those with real Killing spinors. Every simply connected complete spin manifold which carries a parallel spinor belongs to one of the following classes: Calabi-Yau, hyperk\"{a}hler, $\G_2$ and $\Spin(7)$ manifolds. The spin manifold $(M,g)$ admits a real Killing spinor if and only if its Riemannian cone $(\bar{M},\bar{g})$, $\bar{M} = \mathbb{R}_+ \times M$, $\bar{g} = dr^2 + r^2g$, admits a parallel spinor. Thus, four cases, Einstein-Sasakian, 3-Sasakian, 6-dimensional nearly K\"{a}hler, and nearly parallel $\G_2$ manifolds, occur as manifolds with real Killing spinors. Furthermore, Killing spinors also play a crucial role in physics, especially in supergravity and string theories.

There have been several studies about deformations of Killing spinors. The general theory of deformations of Killing spinors was developed by Wang \cite{Wang}. Van Coevering \cite{Coevering} investigated Einstein deformations and associated Killing spinor's behaviors on some 3-Sasakian manifolds. Ohno and Tomihisa \cite{OhnoTomihisa} identified the space of the infinitesimal deformations of Killing spinors on nearly K\"{a}hler 6-manifolds. In this paper, we identify the space of the infinitesimal deformations of Killing spinors on nearly parallel $\G_2$-manifolds. A nearly parallel $\G_2$-manifold is characterized as an almost $\G_2$-manifold $(M^7,\phi,g)$ with $\ast d\phi = \tau_0\phi$ for some $\tau_0 \in \mathbb{R}^{\times}$. It was shown in \cite{FKMS} there is a one-to-one correspondence between nearly parallel $\G_2$-structures and Killing spinors on 7-dimensional spin manifolds. Now, our main result is the following.

\begin{maintheorem} \label{thmA}
Let $(M^7,\phi,g)$ be a compact nearly parallel $\G_2$-manifold and let $\kappa_0$ be a corresponding Killing spinor. Then the space of the infinitesimal deformations of the Killing spinor $\kappa_0$ is isomorphic to the space $\left\{ \gamma \in \Omega^3_{27}M \mid \ast d\gamma = -\tau_0\gamma \right\} \oplus K_+$. Here, $K_+$ is the space of all Killing spinors and the bundle $\wedge^3_{27}M$ is a 27-dimensional $\G_2$-irreducible component of $\wedge^3 M$.
\end{maintheorem}

Alexandrov and Semmelmann \cite{AlexandrovSemmelmann} showed that the space of the nearly parallel $\G_2$-structure $\phi$ is isomorphic to $\left\{ \gamma \in \Omega^3_{27}M \mid \ast d\gamma = -\tau_0\gamma \right\} \oplus \left\{ f_1 \in \Omega^1M \mid \nabla f_1 = -f_1 \mathbin{\lrcorner} \phi \right\}$. For the reason that nearly parallel structures correspond to Killing spinors, Theorem \ref{thmA} is a reproof of their result through Killing spinors.

Using the same method of proving Theorem \ref{thmA}, we can also examine Rarita-Schwinger fields, which can be called spin-$3/2$ version of the harmonic spinor. Rarita-Schwinger fields were introduced by Rarita and Schwinger \cite{RaritaSchwinger}, which are venerable subject in physics. On the other hand, there has been some recent research in mathematics (\cite{BarMazzeo}, \cite{YasushiSemmelmann}, \cite{YasushiTomihisa}, \cite{OhnoTomihisa}, \cite{Wang}). Homma and Semmelmann \cite{YasushiSemmelmann} classified the manifolds with parallel Rarita-Schwinger fields. They also find positive Einstein manifolds admitting Rarita-Schwinger fields, some 8-dimensional symmetric spaces and some algebraic manifolds. More recently, Ohno and Tomihisa \cite{OhnoTomihisa} identified the space of the Rarita-Schwinger fields with the space of the harmonic 3-forms on nearly K\"{a}hler 6-manifolds. In this paper, we identify the space of the Rarita-Schwinger fields on nearly parallel $\G_2$-manifolds.

\begin{maintheorem} \label{thmB}
The space of the Rarita-Schwinger fields is isomorphic to the space $\left\{ \gamma \in \Omega^3_{27}M \; \middle| \; \ast d\gamma = -\frac{\tau_0}{8}\gamma \right\}$ on compact nearly parallel $\G_2$-manifolds $(M^7, \phi, g)$.
\end{maintheorem}

In Section \ref{sec: Preliminalies} we review some notions dealt with in this paper. Notably, in Subsection \ref{subsec: infinitesimal deformations of killing spinors} we discuss deformations of Killing spinors under metric variations, and in Subsection \ref{subsec: nearly parallel G_2}-\ref{subsec: curvature endo} we do the definition and properties of nearly parallel $\G_2$-manifolds. In Subsection \ref{subsec: action of cross product} we prepare the first key tools necessary to prove Theorem \ref{thmA} and \ref{thmB}. In Section \ref{sec: comparison of differential operators}, we construct relations between various differential operators. These are the second key tools. In Section \ref{sec: infinitesimal deformation of Killing spinor} we investigate infinitesimal deformations of Killing spinors using these key tools and thus obtain Theorem \ref{thmA}. As noted above, Theorem \ref{thmA} and \ref{thmB} are shown using similar techniques. Hence, we arrive at Theorem \ref{thmB} in Section \ref{sec: Rarita-Schwinger fields}. Finally, in Section \ref{sec: examples and applications} we show that there are no non-trivial Rarita-Schwinger fields on some concrete nearly parallel $\G_2$-manifolds.

\section{Preliminalies} \label{sec: Preliminalies}
\subsection{Infinitesimal deformations of Killing spinors} \label{subsec: infinitesimal deformations of killing spinors} 
In this subsection, we deal with Killing spinor variations. See \cite{Wang} for detailed concepts. Let $(M,g)$ be an oriented Riemannian manifold with a spin structure $P_{\Spin(n)}M \to P_{\SO(n)}M$, where $P_{\SO(n)}M$ is the bundle of $g$-orthonormal oriented frames and $P_{\Spin(n)}M$ is its (equivariant) double cover. For each $g$-symmetric automorphism $\alpha: TM \to TM$, we define a new metric by
\[
g^{\alpha}(X,Y) \coloneqq g \left(\alpha^{-1}(X),\alpha^{-1}(Y) \right) \quad {\rm for \: all} \; X,Y \in \Gamma(TM).
\]
Similarly, we denote the bundle of $g^{\alpha}$-orthonormal oriented frames by $P^{\alpha}_{\SO(n)}M$ and the double cover of $P^{\alpha}_{\SO(n)}M$ by $P^{\alpha}_{\Spin(n)}M$. Obviously, the symmetric isomorphism $\alpha$ gives an isomorphism $P_{\SO(n)}M \cong P^{\alpha}_{\SO(n)}M$ and further induces the spin structure equivalence $P_{\Spin(n)}M \cong P^{\alpha}_{\Spin(n)}M$. Thus, the respective spinor bundles $S_{1/2}$ and $S^{\alpha}_{1/2}$ corresponding to the spin structures $P_{\Spin(n)}M$ and $P^{\alpha}_{\Spin(n)}M$ are isomorphic, and we write this isomorphism as $\tilde{\alpha}: S_{1/2} \to S_{1/2}^{\alpha}$. Set the Levi-Civita connections for metrics $g$ and $g^{\alpha}$ to $\nabla$ and $\nabla^{\alpha}$, respectively. Then the connection $\bar{\nabla}^{\alpha}$ determined by $\bar{\nabla}^{\alpha} \coloneqq \alpha^{-1} \circ \nabla^{\alpha} \circ \alpha$ is the metric connection of $g$ with torsion.

We assume that $\alpha(t)$ is a smooth path of symmetric automorphism with $\alpha(0) = \Id_{TM}$ and that for each $\alpha(t)$ the metric $g^{\alpha(t)}$ admits a Killing spinor $\kappa_t^{\#}$. We note that if we further assume $\vol(g^{\alpha(t)})=1$, the Killing constant $c$ is independent of variables $t$. Since $\nabla^{\alpha(t)}_X \kappa_t^{\#} = cX \cdot_t \kappa_t^{\#}$ holds, it follows that
\begin{equation} \label{eq: deformation of Killing spinor}
\mathcal{L}^c(\alpha(t), \kappa_t)(X) \coloneqq \bar{\nabla}^{\alpha(t)}_X \kappa_t - c\alpha(t)^{-1}(X) \cdot \kappa_t = 0,
\end{equation}
where $\kappa_t$ is the pull-back path $\widetilde{\alpha(t)}^{-1}(\kappa_t^{\#})$ and $\bar{\nabla}^{\alpha(t)}$ is the pull-back of the metric connection ${\nabla}^{\alpha(t)}$ to the original spinor bundle $S_{1/2}$ defined by $\bar{\nabla}^{\alpha(t)} = \widetilde{\alpha(t)}^{-1} \circ \nabla^{\alpha(t)} \circ \widetilde{\alpha(t)}$. We call a path $(\alpha(t), \kappa_t)$ satisfying the equation $(\ref{eq: deformation of Killing spinor})$ {\it a deformation of the Killing spinor $\kappa_0$}. We use the dot to denote the derivative at $t=0$ of paths $\alpha(t)$ and $\kappa_t$.

Now, we introduce the twisted Dirac operator on $S_{1/2} \otimes TM^{\C}$ by
\[
D_{TM} = \sum_{k=1}^n(e_k \cdot \otimes \Id_{TM^{\C}}) \circ \nabla_{e_k},
\]
where $\{ e_k \}$ is a local orthonormal frame of $TM$. Let the vector bundle $S_{3/2}$ be the kernel of Clifford multiplication $S_{1/2} \otimes TM^{\C} \ni \zeta \otimes X \mapsto X \cdot \zeta \in S_{1/2}$. Then, we have the $\Spin(n)$-decomposition 
\[
S_{1/2} \otimes TM^{\C} \cong S_{1/2} \oplus S_{3/2}.
\]
With respect to this decomposition, we can write $D_{TM}$ as the $2 \times 2$-matrix
\[
D_{TM} = \left(
 \begin{array}{ccc}
 \frac{2-n}{n} D & 2 P^{\ast} \\[1ex]
 \frac{2}{n}P & Q
 \end{array}
\right),
\]
where $P: \Gamma(S_{1/2}) \to \Gamma(S_{3/2})$ is the Penrose operator and $P^{\ast}$ is the formal adjoint operator of $P$. The operator $Q:\Gamma(S_{3/2}) \rightarrow \Gamma(S_{3/2})$ is called {\it the Rarita-Schwinger operator}, which is a formally self-adjoint elliptic differential operator of first order.

For given a spinor $\kappa_0$ and an endomorphism $H: TM \to TM$, we provide a tensor field $\Psi^{(H,\kappa_0)} \in \Gamma(S_{1/2} \otimes T^{\ast}M^{\C})$:
\begin{equation*}
\Psi^{(H,\kappa_0)}(X) = H(X) \cdot \kappa_0 \quad {\rm for \: all} \; X \in \Gamma(TM).
\end{equation*}
Differentiating the equation $(\ref{eq: deformation of Killing spinor})$ at $(\Id,\kappa_0)$, then we have

\begin{proposition}[\cite{Wang}] \label{prop: infinitesimal deformation of Killing spinor by Wang}
Let $(M,g)$ be a Riemannian spin manifold admitting a non-zero Killing spinor $\kappa_0$ with constant $c$. Then we have
\[
d\mathcal{L}^c(\dot{\alpha},\dot{\kappa}) = \nabla_X\dot{\kappa} - cX \cdot \dot{\kappa} + c\dot{\alpha}(X) \cdot \kappa_0 - \frac{1}{2}\sum_i e_i \cdot (\nabla_{e_i}\dot{\alpha})(X) \cdot \kappa_0 + \frac{1}{2}g(\delta\dot{\alpha},X)\kappa_0.
\]
If $M$ is compact and $\dot{\alpha}$ satisfies $\tr\dot{\alpha}=0=\delta\dot{\alpha}$, then $d\mathcal{L}^c(\dot{\alpha},\dot{\kappa})=0$ if and only if $\nabla_X\dot{\kappa}=cX \cdot \dot{\kappa}$ and $D_{TM}\Psi^{(\dot{\alpha},\kappa_0)} = nc\Psi^{(\dot{\alpha},\kappa_0)}$.
\end{proposition}

In view of Proposition \ref{prop: infinitesimal deformation of Killing spinor by Wang}, we call a pair $(H, \kappa)$ {\it an infinitesimal deformation of the Killing spinor $\kappa_0$ with constant $c$} if $H:TM \to TM$ symmetric and $\kappa \in \Gamma(S_{1/2})$ satisfy
 \begin{enumerate}
 \renewcommand{\labelenumi}{(\roman{enumi})}
 \item $\kappa$ is a Killing spinor with constant $c$. 
 \item $\tr H = \delta H=0$.
 \item $D_{TM} \Psi^{(H, \kappa_0)} = nc\Psi^{(h, \kappa_0)}$.
 \end{enumerate}
We define the symmetric tensor $h$ corresponding to the symmetric endomorphism $H$ by $h(X,Y) \coloneqq -2g(H(X),Y)$. It is important that this $h$ is an infinitesimal Einstein deformation (cf. \cite{Wang}).

\subsection{Nearly parallel $\G_2$-manifolds} \label{subsec: nearly parallel G_2}
Let $(e_1, \ldots, e_7)$ be the standard coordinate system on $\mathbb{R}^7$ and $(e^1, \ldots, e^7)$ its dual. We define a 3-form $\phi_0$ on $\mathbb{R}^7$ as
\[
\phi_0 = e^{123}+e^{176}+e^{257}+e^{653}+e^{145}+e^{246}+e^{347},
\]
where wedge signs are omitted. This 3-form $\phi_0$ is called the associative 3-form. The stabilizer of $\phi_0$ under the action of $\GL(7,\mathbb{R})$ on $\wedge^3 \mathbb{R}^7$ is the Lie group $\G_2$:
\[
\G_2 = \left\{ A \in \GL(7,\mathbb{R}) \mid A^{\ast}\phi_0=\phi_0 \right\}.
\]
The group $\G_2$ is a 14-dimensional compact, connected, simply-connected, simple Lie group, which preserves both the metric and orientation on $\mathbb{R}^7$. Therefore it also preserves the Hodge star $\ast$, and especially the 4-form
\[
\psi_0 = \ast \phi_0 = e^{4567}-e^{2345}+e^{1346}-e^{1247}+e^{2367}+e^{1357}+e^{1256}.
\]
Let $M$ be an oriented 7-manifold, and $\phi$ be a 3-form on $M$. At each point $x \in M$, if there exists an orientation preserving isomorphism $T_x M \to \mathbb{R}^7$ that identifies $\phi_x$ and $\phi_0$, we call $\phi$ a (almost) $\G_2$-structure. We remark that $\G_2$-structures correspond one-to-one with reductions of the structure group of $M$ to $\G_2$. This 3-form $\phi$ determines a Riemannian metric $g$ via the relation
\[
(X \mathbin{\lrcorner} \phi) \wedge (Y \mathbin{\lrcorner} \phi) \wedge \phi = -6g(X,Y)\vol \quad {\rm for \: all} \; X,Y \in \Gamma(TM),
\]
where $\vol$ is a volume form of $g$, and this orientation is compatible with the original one. Our choice of the orientation is the opposite of Bryant \cite{Bryant}, but the same as Alexandrov and Semmelmann \cite{AlexandrovSemmelmann}. Furthermore, a $\G_2$-structure $\phi$ induces a cross product $A: TM \times TM \to TM$ and a tangent valued 3-form $\chi$ on $M$ by
\[
g(X,A(Y,Z)) = \phi(X,Y,Z), \quad \frac{1}{2}g(X,\chi(Y,Z,W)) = \psi(X,Y,Z,W),
\]
for all $X,Y,Z,W \in \Gamma(TM)$. We also define a endomorphism $A_Y$ by $A_Y Z \coloneqq A(Y,Z)$. A $\G_2$-structure $\phi$ is called {\it a nearly parallel $\G_2$-structure} if
\begin{equation} \label{eq: def of nearly parallel G2}
\ast d\phi = \tau_0 \phi,
\end{equation}
for some nonzero constant $\tau_0$. Alexandrov and Semmelmann \cite{AlexandrovSemmelmann} examined conditions that are equivalent to the equation $(\ref{eq: def of nearly parallel G2})$.
\begin{proposition}[{\cite[Proposition 2.4]{AlexandrovSemmelmann}}]
Let $(M,\phi,g)$ be an almost $\G_2$-manifold. Then the following conditions are equivalent to each other:
 \begin{enumerate}
 \renewcommand{\labelenumi}{(\roman{enumi})}
 \item There exists a $\tau_0 \in \mathbb{R} \backslash \{0\}$ with $\ast d\phi = \tau_0 \phi$.
 \item There exists a $\tau_0 \in \mathbb{R} \backslash \{0\}$ with $\nabla \phi = \frac{\tau_0}{4}\ast \phi$.
 \item There exists a $\tau_0 \in \mathbb{R} \backslash \{0\}$ with $\nabla_X \psi = - \frac{\tau_0}{4}X \wedge \phi$ for all vector fields $X$.
 \end{enumerate}
\end{proposition}
We call a 7-dimensional manifold $M$ with a nearly parallel $\G_2$-structure $\phi$ a nearly parallel $\G_2$-manifold. The canonical connection $\bar{\nabla}$ of a $\G_2$-structure, defined by
\begin{equation} \label{eq: G2 conn on tangent bundle}
\bar{\nabla}_X Y \coloneqq \nabla_X Y - \frac{\tau_0}{12}A_X Y \quad {\rm for \: all} \; X,Y \in \Gamma(TM),
\end{equation}
satisfies $\bar{\nabla}g=0$ and $\bar{\nabla}\phi=0$, and thus $\bar{\nabla}A=0$. Extending the canonical connection $\bar{\nabla}$ to a tensor bundle $\mathcal{T}M$ yields the relation $\bar{\nabla}_X = \nabla_X - \frac{\tau_0}{12}A_{X \star}$. Here, we denote by $A_{X \star}$ in $\End \mathcal{T}M$ the natural extension of $A_X$ which is a endomorphism of the tangent bundle (cf. \cite[p.3059]{MoroianuSemmelmann}). From now on, we identify $TM$ with $T^{\ast}M$ using the metric without notice.

As stated in \cite{FKMS}, on 7-dimensional spin manifolds, there is a one-to-one correspondence between nearly parallel $\G_2$-structures and unit Killing spinors with the Killing number $\frac{\tau_0}{8}$ in the real spinor bundle up to sign. Moreover, nearly parallel $\G_2$-manifolds which are simply connected and not of constant curvature are classified into three types depending on the dimension of the space $[KS]$ of all Killing spinors:
\begin{center}
nearly parallel $\G_2$-manifolds of type 1: $\dim[KS]=1$, \\
nearly parallel $\G_2$-manifolds of type 2: $\dim[KS]=2$, \\
nearly parallel $\G_2$-manifolds of type 3: $\dim[KS]=3$.
\end{center}
We also call a manifold of type 1 a proper nearly parallel $\G_2$-manifold. For nearly parallel $\G_2$-manifolds of type 2 and type 3, independent Killing spinors constitute an Einstein-Sasakian structure (but not a 3-Sasalian structure) and a 3-Sasakian structure, respectively (cf. \cite{BFGK}). We note that on any 7-dimensional spin manifold, the complex spinor bundle $S_{1/2}$ coincides with the complexification of the real spinor bundle $\slashed{S}$. It is an important fact that the Ricci curvature satisfies $\Ric = 24\left(\frac{\tau_0}{8}\right)^2 g$ on nearly parallel $\G_2$-manifolds.

The spin connection $\nabla$ is defined as the Levi-Civita connection pulled back to the spinor bundle. Similarly, we pull back the canonical connection of a $\G_2$-structure to the spinor bundle and also denote this connection as $\bar{\nabla}$, which is explicitly expressed as
\begin{equation} \label{eq: G2 conn on spinor bundle}
\bar{\nabla}_X \zeta = \nabla_X \zeta - \frac{\tau_0}{24} (X \mathbin{\lrcorner} \phi) \cdot \zeta \quad {\rm for \: all} \; X \in \Gamma(TM), \zeta \in \Gamma(S_{1/2}).
\end{equation}

\subsection{Algebraic results on nearly parallel $\G_2$-manifolds} \label{subsec: alg results on nearly G2}
Assume that $(M,\phi,g)$ is a nearly parallel $\G_2$-manifold with the normalized scalar curvature $\scal = 42$ (i.e. $\tau_0 = 4$). We decompose some exterior bundles and the spinor bundle (for example, see \cite{Karigiannis}).

The Lie group $\G_2$ acts as the subgroup of $\SO(7)$ on the space of differential forms. The bundle $\wedge^p M$ is irreducible for $p=0,1,6,7$, but reducible for $p=2,3,4,5$. In fact, we have $\G_2$-irreducible decompositions
\begin{align*}
\wedge^2M &\cong \wedge^2_7M \oplus \wedge^2_{14}M, \\
\wedge^3M &\cong \wedge^3_1M \oplus \wedge^3_7M \oplus \wedge^3_{27}M,
\end{align*}
where $\wedge^p_dM$ is the $d$-dimensional $\G_2$-irreducible component. Because of the Hodge duality $\wedge^k M \cong \ast \wedge^{7-k}M$, the decomposition of the bundles $\wedge^4M$ and $\wedge^5M$ are obtained from the decompositions of the bundles $\wedge^3M$ and $\wedge^2M$, respectively. We remark that the canonical $\G_2$-connection $\bar{\nabla}$ preserves these decompositions. These irreducible components are characterized as follows
\begin{align*}
\wedge^2_7M &= \{ X\mathbin{\lrcorner}\phi \in \wedge^2M \mid X \in TM \} \\
&= \{ \beta \in \wedge^2M \mid \ast (\phi \wedge \beta) = -2\beta \}, \\
\wedge^2_{14}M &= \{ \beta \in \wedge^2M \mid \beta \wedge \psi = 0 \} \\
&= \{ \beta \in \wedge^2M \mid \ast (\phi \wedge \beta) = \beta \}, \\
\wedge^3_1M &= \langle \phi \rangle, \\
\wedge^3_7M &= \{ X\mathbin{\lrcorner}\psi \in \wedge^3M \mid X \in TM \} \\
&= \{ \alpha \in \wedge^3M \mid \alpha(X,Y,A_X Y)=0, \forall X,Y \in TM \}, \\
\wedge^3_{27}M &= \{ \alpha \in \wedge^3M \mid \alpha \wedge \phi = 0, \alpha \wedge \psi = 0 \}.
\end{align*}
The map $\mathbf{i}: H \mapsto -2H_{\star}\phi$ defines an isomorphism between $\Sym_0M$ and $\wedge^3_{27}M$, where $\Sym_0M$ is the trace-free part of the bundle of symmetric endomorphisms  $\Sym M$. On the other hand, the inverse map $\mathbf{j}=-8\mathbf{i}^{-1}$ is given by
\[
\mathbf{j}(\gamma)(X,Y) = \ast((X \mathbin{\lrcorner} \phi) \wedge (Y \mathbin{\lrcorner} \phi) \wedge \gamma).
\]

We denote by $\kappa_0$ the Killing spinor corresponding to the nearly parallel $\G_2$-structure $\phi$. This Killing spinor $\kappa_0$ defines the isomorphism $\sigma \mapsto \sigma \cdot \kappa_0$ from $\wedge^0M \oplus \wedge^1M$ to $\slashed{S}$. Under this decomposition, for any spinor $\zeta=f\kappa_0 + \alpha \cdot \kappa_0 \in \slashed{S}$, we write $\zeta = (f,\alpha) \in \wedge^0M \oplus \wedge^1M$. As shown in \cite{Karigiannis}, the Clifford multiplication of a 1-form $Y$ to a real spinor $(f,\alpha)$ is given by octonion multiplication,
\begin{equation} \label{eq: clifford action}
Y \cdot (f,\alpha) = (-g(Y,\alpha), fY + A_Y \alpha).
\end{equation}

We note down some algebraic relations for later use.
\begin{lemma}
The following hold for all vector fields $X,Y,Z$:
\begin{gather}
A_X A_Y Z = -g(X,Y)Z + g(X,Z)Y - \frac{1}{2}\chi(X,Y,Z), \label{eq1} \\
2A_{A_XY}Z = A_{A_YZ}X + A_{A_ZX}Y + 3g(X,Z)Y - 3g(Y,Z)X, \label{eq2} \\
(X \mathbin{\lrcorner} Y \mathbin{\lrcorner} \phi) \mathbin{\lrcorner} \phi + X \mathbin{\lrcorner} Y \mathbin{\lrcorner} \psi = - X \wedge Y, \label{eq9} \\
\phi \wedge X \wedge Y = \ast(Y \mathbin{\lrcorner} X \mathbin{\lrcorner} \psi). \label{eq12}
\end{gather}
\end{lemma}

These formulas were calculated in various papers. See, for example, \cite{SalamonWalpuski}.

\subsection{The curvature endomorphism} \label{subsec: curvature endo}
Let $(M,\phi,g)$ be a nearly parallel $\G_2$-manifold with $\scal = 42$. In this subsection, we consider several notions of curvatures. In particular, we introduce properties of operators of the curvature for the canonical $\G_2$-connection $\bar{\nabla}$. Note that in the remaining part of this paper, we adopt the Einstein convention of summation on the repeated subscripts.

First, recall that the definition of the basic curvature tensor:
\[
R(W,X)Y = \nabla_W \nabla_X Y - \nabla_X \nabla_W Y - \nabla_{[W,X]}Y \quad {\rm for \: all} \; X,Y \in \Gamma(TM).
\]
Replacing the Levi-Civita connection with the canonical $\G_2$-connection, we define the curvature tensor $\bar{R}$ for the connection $\bar{\nabla}$. The difference between these two curvature tensors $R$ and $\bar{R}$ is known by the following lemma.

\begin{lemma}[\cite{AlexandrovSemmelmann}] \label{lem: difference curvatures}
For any vector fields $W,X,Y$, we have
\begin{equation*}
\begin{split}
\bar{R}(W,X)Y - R(W,X)Y &= \frac{1}{9} \left\{ 2A_{A_W X}Y + A_{A_X Y}W + A_{A_Y W}X \right\} \\
&= \frac{1}{9} \left\{ 4A_{A_W X}Y -3g(W,Y)X + 3g(X,Y)W \right\}.
\end{split}
\end{equation*}
\end{lemma}

Since $(M,\phi,g)$ is an Einstein manifold with the Einstein constant 6, we immediately see that the Ricci curvature for the canonical $\G_2$-connection satisfies $\overline{\Ric} = \frac{16}{3}g$. Computing with the equation $(\ref{eq1})$, $(\ref{eq2})$ and Lemma \ref{lem: difference curvatures}, we obtain a formula that can be called the 1st Bianchi identity for the canonical $\G_2$-connection $\bar{\nabla}$:
\begin{equation} \label{eq: 1st Bianchi identity G2 ver}
\bar{R}(W,X)Y + \bar{R}(X,Y)W + \bar{R}(Y,W)X = \frac{2}{3}\chi(W,X,Y).
\end{equation}
Furthermore, for any vector fields $W,X,Y,Z$, the curvature tensor $\bar{R}$ satisfies the alternations and symmetries
\[
\bar{R}(W,X,Y,Z) = -\bar{R}(X,W,Y,Z) = -\bar{R}(W,X,Z,Y) = \bar{R}(Y,Z,W,X), 
\]
where we denote $g(\bar{R}(W,X)Y,Z)$ by $\bar{R}(W,X,Y,Z)$. Next, we introduce the curvature operator $R: \wedge^2M \to \wedge^2M$:
\[
R(e_i \wedge e_j) = \frac{1}{2}R_{ijkl}e_k \wedge e_l,
\]
where $\{ e_i \}$ is a local orthonormal frame of $TM$. Let $EM$ be a vector bundle associated to the oriented orthonormal frame bundle or the spinor bundle on $(M,g)$. With the curvature operator $R$, we define the curvature endomorphism $q(R) = q_E(R) \in \End EM$ as
\begin{equation*}
q(R) = \frac{1}{2}(e_i \wedge e_j)_{\star}R(e_i \wedge e_j)_{\star}.
\end{equation*}
Similarly, we define the curvature endomorphism $q(\bar{R})$ for the canonical $\G_2$-connection. In particular, on the space of 1-forms, we have $q_T(R) = \Ric =6\Id$ and $q_T(\bar{R}) = \overline{\Ric} = \frac{16}{3}\Id$. An important property of the curvature endomorphism $q(\bar{R})$ is that it preserves all tensor bundles associated to $\G_2$-representations. By noting the isomorphism $\wedge^2_{14}\mathbb{R}^7 \cong \mathfrak{g}_2$, we can show this property as in \cite[p. 3061]{MoroianuSemmelmann}.

\subsection{The action of cross product} \label{subsec: action of cross product}
As we have already mentioned lightly in Subsection \ref{subsec: nearly parallel G_2},  for any endomorphism or 2-tensor $B$ in $\End TM \cong T^{\ast}M \otimes TM$, we denote by $B_{\star}$ the natural extension on the tensor bundle $\mathcal{T}M$. In particular, the induced action of $B$ on $p$-form $u$ is written as
\[
B_{\star}u = -B^{\ast}(e_i) \wedge e_i \mathbin{\lrcorner} u,
\] 
where $B^{\ast}$ is the metric adjoint of $B$ and $\{ e_i \}$ is a local orthonormal frame of $TM$. Let $w$ and $H$ be elements of $\wedge^2_{14}M$ and $\Sym_0M$, respectively, and $A$ be the cross product defined in Subsection \ref{subsec: nearly parallel G_2}. Then, for any tangent vector $X$, we have 
\begin{gather} 
\ast(H_{\star}\phi) = -H_{\star}\psi, \quad \ast(H_{\star}\psi) = -H_{\star}\phi, \label{eq: sym action to associative} \\
A_{X \star}\phi = 3X \mathbin{\lrcorner} \psi, \quad A_{X \star}\psi = -3 X \wedge \phi, \label{eq: cross product action to associative} \\
w_{\star}\phi = 0 \label{eq: wedge^2_14 action to associative}.
\end{gather}
\begin{proof}
First, we note that $H_{\star}\phi$ is in $\wedge^3_{27}M$. For any element $\alpha$ of $\wedge^3_{27}M$, we get
\begin{equation*}
\begin{split}
\langle \alpha, H_{\star}\phi \rangle \vol &= -\langle \alpha, He_i \wedge e_i \mathbin{\lrcorner} \phi \rangle \vol = -\langle e_i \wedge He_i \mathbin{\lrcorner} \alpha, \phi \rangle \vol \\
&= \langle H_{\star}\alpha, \phi \rangle \vol = (H_{\star}\alpha) \wedge \ast\phi = H_{\star}(\alpha \wedge \psi) - \alpha \wedge (H_{\star}\psi) \\
&= -\alpha \wedge \ast^2(H_{\star}\psi) = \langle \alpha, -\ast(H_{\star}\psi) \rangle \vol.
\end{split}
\end{equation*}
Next, we show the first formula of $(\ref{eq: cross product action to associative})$. For any tangent vectors $Y,Z,W$, we have
\begin{eqnarray*}
(A_{X \star}\phi)(Y,Z,W) &=& - g(A_X Y,A_Z W) - g(Y,A_{A_X Z}W) - g(Y,A_Z A_X W) \\
&=& g(Y, A_X A_Z W + A_W A_X Z -A_Z A_X W) \\
&\overset{(\ref{eq2})}{=}& -3g(Y, A_Z A_X W) + 3g(Y,X)g(W,Z) - 3g(Y,W)g(X,Z) \\
&\overset{(\ref{eq1})}{=}& -3g \left(Y, -g(Z,X)W + g(Z,W)X - \frac{1}{2}\chi(Z,X,W) \right) \\
&& + 3g(Y,X)g(W,Z) - 3g(Y,W)g(X,Z) \\
&=& 3\psi(Y,Z,X,W) = 3X \mathbin{\lrcorner} \psi(Y,Z,W).
\end{eqnarray*}
The second formula of $(\ref{eq: cross product action to associative})$ follows from the fact that the equation $\ast(A_{X \star}\phi) = A_{X \star}\psi$ holds. \\
Finally, since the isomorphism $\wedge^2_{14}\mathbb{R}^7 \cong \mathfrak{g}_2$ holds, we know that $w_{\star}\phi$ is in $\wedge^3_1M$. Thus $w_{\star}\phi \in \wedge^3_1M$ is expressed as $a\phi$ by using some constant $a$. We just need to substitute the basis and show the constant $a=0$.
\end{proof}

By applying Schur's Lemma to $\G_2$-decompositions, we get some relations.

\begin{lemma}
The following relations hold:
\begin{eqnarray}
e_i \wedge (A_{e_i \star}w) &=& 0 \quad {\rm for \: all} \; w \in \wedge^2_{14}M, \label{Schur w} \\
e_i \mathbin{\lrcorner} (A_{e_i \star}w) &=& 0 \quad {\rm for \: all} \; w \in \wedge^2_{14}M, \label{Schur w'} \\
(A_{e_i \star}H)(e_i) &=& 0 \quad {\rm for \: all} \; H \in \Sym_0 M, \label{Schur H} \\
e_i \mathbin{\lrcorner} (A_{e_i \star}\gamma) &=& 0 \quad {\rm for \: all} \; \gamma \in \wedge^3_{27} M. \label{Schur gamma}
\end{eqnarray}
\end{lemma}

The following, Proposition \ref{prop: cross product action1} and Proposition \ref{prop: cross product action2}, work essentially in proving Theorem \ref{thmA} and Theorem \ref{thmB}. Before that, we introduce the lemma.

\begin{lemma}[\cite{AlexandrovSemmelmann}]
Let $H$ be a section of $\Sym_0 M$. A section $\gamma$ is defined by $\gamma = \mathbf{i}(H) = -2H_{\star}\phi$. Then we have 
\begin{equation} \label{eq: A_{e_i}gamma}
A_{e_i \star}\bar{\nabla}_{e_i}\gamma = -3 (\ast \bar{d}\gamma)_{\wedge^3_7} + (\ast \bar{d}\gamma)_{\wedge^3_{27}} = 3\delta H \mathbin{\lrcorner} \psi + (\ast d\gamma)_{\wedge^3_{27}} + \frac{2}{3}\gamma.
\end{equation}
\end{lemma}

\begin{proposition} \label{prop: cross product action1}
Let $w$ and $H$ be sections of $\wedge^2_{14}M$ and $\Sym_0 M$, respectively. A section $\gamma$ is defined by $\gamma = \mathbf{i}(H) = -2H_{\star}\phi$. Then we have
\begin{align}
A_{e_i \star}\bar{\nabla}_{e_i}w &= \delta w \mathbin{\lrcorner} \phi, \label{cross product action1 wedge^2_{14}} \\
\mathbf{i}(A_{e_i \star}\bar{\nabla}_{e_i}H) &= -2(\ast \bar{d}\gamma)_{\wedge^3_{27}} = -2(\ast d\gamma)_{\wedge^3_{27}} - \frac{4}{3}\gamma \label{cross product action1 Sym_0}
\end{align} 
\end{proposition}
\begin{proof}
For any vector field $X$, we get
\begin{eqnarray*}
(A_{e_i \star}\bar{\nabla}_{e_i}w)(X) &=& A_{e_i} (\bar{\nabla}_{e_i}w)(X) - (\bar{\nabla}_{e_i}w)(A_{e_i}X) \\
&=& (\bar{\nabla}_{e_i}w)(X) \mathbin{\lrcorner} e_i \mathbin{\lrcorner} \phi - (\bar{\nabla}_{e_i}w)(X \mathbin{\lrcorner} e_i \mathbin{\lrcorner} \phi) \\
&=& (\bar{\nabla}_{e_i}w)(e_j,X)e_i \mathbin{\lrcorner} e_j \mathbin{\lrcorner} \phi - \phi(e_j,e_i,X)(\bar{\nabla}_{e_i}w)e_j \\
&=& ((\bar{\nabla}_{e_i}w)e_j \wedge (e_i \mathbin{\lrcorner} e_j \mathbin{\lrcorner} \phi))(X).
\end{eqnarray*}
Thus, we have
\[
A_{e_i \star}\bar{\nabla}_{e_i}w = (\bar{\nabla}_{e_i}w)e_j \wedge (e_i \mathbin{\lrcorner} e_j \mathbin{\lrcorner} \phi).
\]
We proceed further with the calculation.
\begin{eqnarray*}
A_{e_i \star}\bar{\nabla}_{e_i}w &=& -e_i \mathbin{\lrcorner} ((\bar{\nabla}_{e_i}w)e_j \wedge e_j \mathbin{\lrcorner} \phi) + (e_i \mathbin{\lrcorner} (\bar{\nabla}_{e_i}w)e_j)e_j \mathbin{\lrcorner} \phi \\
&\overset{(\ref{Schur w'})}{=}& -e_i \mathbin{\lrcorner} \bar{\nabla}_{e_i}(w(e_j) \wedge e_j \mathbin{\lrcorner} \phi) + \delta w \mathbin{\lrcorner} \phi \\
&=& -e_i \mathbin{\lrcorner} \bar{\nabla}_{e_i}(w_{\star}\phi) + \delta w \mathbin{\lrcorner} \phi \\
&\overset{(\ref{eq: wedge^2_14 action to associative})}{=}& \delta w \mathbin{\lrcorner} \phi.
\end{eqnarray*}
Next, we use the isomorphism $\mathbf{i}: \Sym_0 M \to \wedge^3_{27}M$ defined in Subsection \ref{subsec: alg results on nearly G2}.
\begin{eqnarray*}
\mathbf{i}(A_{e_i \star}\bar{\nabla}_{e_i}H) &=& -2(A_{e_i \star}\bar{\nabla}_{e_i}H)_{\star}\phi \\
&=& -2A_{e_i \star}((\bar{\nabla}_{e_i}H)_{\star}\phi) + 2(\bar{\nabla}_{e_i}H)_{\star}(A_{e_i \star}\phi) \\
&\overset{(\ref{eq: cross product action to associative})}{=}& -2A_{e_i \star}\bar{\nabla}_{e_i}(H_{\star}\phi) + 6(\bar{\nabla}_{e_i}H)_{\star}(e_i \mathbin{\lrcorner} \psi) \\
&=& A_{e_i \star}\bar{\nabla}_{e_i}\gamma + 6(\bar{\nabla}_{e_i}H)e_i \mathbin{\lrcorner} \psi + 6e_i \mathbin{\lrcorner} \bar{\nabla}_{e_i}(H_{\star}\psi) \\
&\overset{(\ref{eq: sym action to associative}),(\ref{Schur H})}{=}& A_{e_i \star}\bar{\nabla}_{e_i}\gamma - 6\delta H \mathbin{\lrcorner} \psi + 3e_i \mathbin{\lrcorner} \bar{\nabla}_{e_i}(\ast \gamma) \\
&=& A_{e_i \star}\bar{\nabla}_{e_i}\gamma - 6\delta H \mathbin{\lrcorner} \psi - 3\ast \bar{d}\gamma \\
&\overset{(\ref{eq: A_{e_i}gamma})}{=}& 3\delta H \mathbin{\lrcorner} \psi + (\ast \bar{d}\gamma)_{\wedge^3_{27}} - 6\delta H \mathbin{\lrcorner} \psi - 3(\ast \bar{d}\gamma)_{\wedge^3_7} - 3(\ast \bar{d}\gamma)_{\wedge^3_{27}} \\
&=& -2(\ast \bar{d}\gamma)_{\wedge^3_{27}} = -2(\ast d\gamma)_{\wedge^3_{27}} - \frac{4}{3}\gamma.
\end{eqnarray*}
\end{proof}

We consider the natural extension $A_{X \star}$ of the cross product $A_X$ to 2-tensor fields in Proposition \ref{prop: cross product action1}. This extension $A_{X \star}$ is expressed for a 2-tensor field $\alpha \otimes \beta \in \Gamma(TM \otimes TM)$ as
\[
A_{X \star}(\alpha \otimes \beta) = A_X\alpha \otimes \beta + \alpha \otimes A_X \beta.
\]
Now, we introduce another extension $\widetilde{A_{X \star}}$ of the cross product $A_X$ to 2-tensor fields as follows:
\[
\widetilde{A_{X \star}}(\alpha \otimes \beta) = A_X\alpha \otimes \beta - \alpha \otimes A_X \beta.
\]

\begin{proposition} \label{prop: cross product action2}
Let $w$, $H$, and $\gamma$ be sections satisfying the assumptions of Proposition \ref{prop: cross product action1}. Then we have
\begin{align}
\widetilde{A_{e_i \star}}\bar{\nabla}_{e_i}H &= \frac{1}{2} \bar{\delta} \gamma - \delta H \mathbin{\lrcorner} \phi = -\frac{1}{3}\delta H \mathbin{\lrcorner} \phi + \frac{1}{2}(\delta\gamma)_{\wedge^2_{14}}, \label{cross product action2 Sym_0} \\
\mathbf{i}(\widetilde{A_{e_i \star}}\bar{\nabla}_{e_i}w) &= 8dw - 2\delta w \mathbin{\lrcorner} \psi. \label{cross product action2 wedge^2_{14}}
\end{align}
\end{proposition}
\begin{proof}
We follow the same procedure as we show the equation $(\ref{cross product action1 wedge^2_{14}})$ to prove the equation $(\ref{cross product action2 Sym_0})$.
\begin{eqnarray*}
\widetilde{A_{e_i \star}}\bar{\nabla}_{e_i}H &=& (\bar{\nabla}_{e_i}H)e_j \wedge (e_i \mathbin{\lrcorner} e_j \mathbin{\lrcorner} \phi) \\
&=& -e_i \mathbin{\lrcorner} ((\bar{\nabla}_{e_i}H)e_j \wedge e_j \mathbin{\lrcorner} \phi) + (\bar{\nabla}_{e_i}H)(e_j,e_i)e_j \mathbin{\lrcorner} \phi \\
&\overset{(\ref{Schur H})}{=}& -e_i \mathbin{\lrcorner} \bar{\nabla}_{e_i}(He_j \wedge e_j \mathbin{\lrcorner} \phi) - \delta H \mathbin{\lrcorner} \phi \\
&=& e_i \mathbin{\lrcorner} \bar{\nabla}_{e_i}(H_{\star}\phi) - \delta H \mathbin{\lrcorner} \phi \\
&=& - \frac{1}{2}e_i \mathbin{\lrcorner} \bar{\nabla}_{e_i}\gamma - \delta H \mathbin{\lrcorner} \phi \\
&=& \frac{1}{2} \bar{\delta} \gamma - \delta H \mathbin{\lrcorner} \phi.
\end{eqnarray*}
From the characterization of the bundle $\wedge^2_7M$, the $\wedge^2_7M$-part $\left( \bar{\delta} \gamma \right)_{\wedge^2_7}$ of $\bar{\delta} \gamma$ is denoted by $\left( \bar{\delta} \gamma \right)_{\wedge^2_7} = X \mathbin{\lrcorner} \phi$ for some vector field $X$. This $X$ is obtained by taking the inner product with any vector field $Z$ as follows:
\begin{align*}
3\langle Z,X \rangle &= \langle Z \mathbin{\lrcorner} \phi, X \mathbin{\lrcorner} \phi \rangle = \langle Z \mathbin{\lrcorner} \phi, \bar{\delta}\gamma \rangle = -\frac{1}{2} (\bar{\nabla}_{e_i}\gamma)(e_i,e_j,e_k) \phi(Z,e_j,e_k) \\
&= 4\bar{\delta}H(Z) \overset{(\ref{Schur H})}{=} 4\langle Z, \delta H \rangle.
\end{align*}
Thus, we get $\left( \bar{\delta} \gamma \right)_{\wedge^2_7} = \frac{4}{3}\delta H \mathbin{\lrcorner} \phi$. For the $\wedge^2_{14}M$-part $\left( \bar{\delta}\gamma \right)_{\wedge^2_{14}}$ of $\bar{\delta} \gamma$, the equation $(\ref{Schur gamma})$ gives $\left( \bar{\delta}\gamma \right)_{\wedge^2_{14}} = \left( \delta\gamma \right)_{\wedge^2_{14}}$. \\
Next, since $\widetilde{A_{e_i \star}}\bar{\nabla}_{e_i}w = -A_{e_i \star}\bar{\nabla}_{e_i}w - 2\bar{\nabla}_{e_i}w \circ A_{e_i}$ holds, we have
\begin{eqnarray*}
\mathbf{i}(\widetilde{A_{e_i \star}}\bar{\nabla}_{e_i}w) &=& -2(\widetilde{A_{e_i \star}}\bar{\nabla}_{e_i}w)_{\star}\phi \\
&=& 2(A_{e_i \star}\bar{\nabla}_{e_i}w)_{\star}\phi + 4((\bar{\nabla}_{e_i}w) \circ A_{e_i})_{\star}\phi \\
&\overset{(\ref{cross product action1 wedge^2_{14}})}{=}& 2(\delta w \mathbin{\lrcorner} \phi)_{\star}\phi - 4(A_{e_i}\bar{\nabla}_{e_i}w(e_k)) \wedge e_k \mathbin{\lrcorner} \phi \\
&=& 2A_{\delta w \star}\phi - 4A_{e_i \star}((\bar{\nabla}_{e_i}w)(e_k) \wedge e_k \mathbin{\lrcorner} \phi) \\
&& + 4(\bar{\nabla}_{e_i}w)(e_k) \wedge A_{e_i \star}(e_k \mathbin{\lrcorner} \phi) \\ 
&\overset{(\ref{eq: cross product action to associative})}{=}& 6\delta w \mathbin{\lrcorner} \psi -4A_{e_i \star}\bar{\nabla}_{e_i}(w_{\star}\phi) \\
&&+ 4(\bar{\nabla}_{e_i}w)(e_k) \wedge (A_{e_i}e_k \mathbin{\lrcorner} \phi + e_k \mathbin{\lrcorner} A_{e_i \star}\phi) \\
&\overset{(\ref{eq: cross product action to associative}),(\ref{eq: wedge^2_14 action to associative})}{=}& 6\delta w \mathbin{\lrcorner} \psi + 4(\bar{\nabla}_{e_i}w)(e_k) \wedge A_{A_{e_i}e_k} + 12(\bar{\nabla}_{e_i}w)(e_k) \wedge e_k \mathbin{\lrcorner} (e_i \mathbin{\lrcorner} \psi) \\
&=& 6\delta w \mathbin{\lrcorner} \psi + 4(\bar{\nabla}_{e_i}w)(e_k) \wedge (e_i \wedge e_k) + 8(\bar{\nabla}_{e_i}w)(e_k) \wedge e_k \mathbin{\lrcorner} e_i \mathbin{\lrcorner} \psi \\
&=& 6\delta w \mathbin{\lrcorner} \psi + 8e_i \wedge \bar{\nabla}_{e_i}w + 8(\bar{\nabla}_{e_i}w)_{\star}(e_i \mathbin{\lrcorner} \psi) \\
&\overset{(\ref{Schur w})}{=}& 6\delta w \mathbin{\lrcorner} \psi + 8dw + 8(\bar{\nabla}_{e_i}w)e_i \mathbin{\lrcorner} \psi + 8e_i \mathbin{\lrcorner} (\bar{\nabla}_{e_i} w)_{\star}\psi \\
&\overset{(\ref{eq: wedge^2_14 action to associative}),(\ref{Schur w'})}{=}& 8dw - 2\delta w \mathbin{\lrcorner} \psi.
\end{eqnarray*}
\end{proof}

\section{Comparison of differential operators} \label{sec: comparison of differential operators}
Assume that $(M,\phi,g)$ is a nearly parallel $\G_2$-manifold with $\scal=42$. We defined the twisted Dirac operator $D_{TM}$ in Subsection \ref{subsec: infinitesimal deformations of killing spinors}. We also have a natural second-order differential operator called the standard Laplace operator $\Delta = \nabla^{\ast}\nabla + q(R)$. Similarly, we introduce the twisted Dirac operator $\overline{D_{TM}} = (e_k \cdot \otimes \Id_{TM^{\C}}) \circ \bar{\nabla}_{e_k}$ for the canonical $\G_2$-connection and the $\G_2$-Laplace operator $\bar{\Delta} = \bar{\nabla}^{\ast}\bar{\nabla} + q(\bar{R})$. These Laplace operators are introduced in \cite{SemmelmannWeingart}. To clarify the relationship between the twisted Dirac operator and the $\G_2$-Laplace operator, we proceed.

First, we get the difference between the twisted Dirac operator for the Levi-Civita connection and the one for the canonical $\G_2$-connection.

\begin{theorem}
On the sections of $S_{1/2} \otimes TM^{\C}$, the relation holds:
\begin{equation} \label{eq: twisted dirac relation}
\overline{D_{TM}} = D_{TM} - \frac{1}{2}\phi \cdot \otimes \Id - \frac{1}{3} e_i \cdot \otimes A_{e_i}.
\end{equation}
\end{theorem}
\begin{proof}
A direct calculation using the equation $(\ref{eq: G2 conn on tangent bundle})$ and $(\ref{eq: G2 conn on spinor bundle})$ gives the equation $(\ref{eq: twisted dirac relation})$
\end{proof}

Next, we see the relationship between the $\G_2$-Laplace operator and the twisted Dirac operator for the canonical $\G_2$-connection.

\begin{lemma}
For any vector field $X$ and spinor $\zeta$, we have
\begin{equation} \label{eq: Lichnerowicz lemma}
e_j \cdot \bar{R}_S(X,e_j)\zeta = -\frac{1}{2}\overline{\Ric}(X) \cdot \zeta - \frac{2}{3}(X\mathbin{\lrcorner}\psi) \cdot \zeta,
\end{equation}
where $\bar{R}_S$ denotes the curvature tensor for the canonical $\G_2$-connection on the spinor bundle $S_{1/2}$.
\end{lemma}
\begin{proof}
Using the Clifford relation, we obtain
\begin{equation*}
 \begin{split}
 12 e_j \cdot \bar{R}_{S}(X,e_j) &= 3g(\bar{R}(X, e_j)e_k,e_l)e_je_ke_l = 3g(X,e_i)\bar{R}_{ijkl}e_je_ke_l \\
 &= g(X,e_i)(\bar{R}_{ijkl}e_je_ke_l + \bar{R}_{iklj}e_ke_le_j + \bar{R}_{iljk}e_le_je_k) \\
 &= g(X,e_i)\bar{R}_{ijkl}e_je_ke_l + g(X,e_i)\bar{R}_{iklj}(-2\delta_{lj}e_k + 2\delta_{kj}e_l + e_je_ke_l) \\
 &\quad + g(X,e_i)\bar{R}_{iljk}(-2\delta_{lj}e_k + 2\delta_{lk}e_j + e_je_ke_l) \\
 &= g(X,e_i)(\bar{R}_{ijkl} + \bar{R}_{iklj} + \bar{R}_{iljk})e_je_ke_l - 6\overline{\Ric}(X,e_k)e_k.
 \end{split}
\end{equation*}
From the 1st Bianchi identity $(\ref{eq: 1st Bianchi identity G2 ver})$ for the canonical $\G_2$-connection $\bar{\nabla}$, it follows that
\begin{equation*}
\begin{split}
&g(X,e_i)(\bar{R}_{ijkl} + \bar{R}_{iklj} + \bar{R}_{iljk})e_je_ke_l = -g(X,e_i)(\bar{R}_{jkli} + \bar{R}_{klji} + \bar{R}_{ljki})e_je_ke_l \\
&= -\frac{2}{3}g(X,e_i)g(\chi(e_j,e_k,e_l),e_i)e_je_ke_l = -\frac{4}{3}\psi(X,e_j,e_k,e_l)e_je_ke_l = -8X\mathbin{\lrcorner}\psi.
\end{split}
\end{equation*}
\end{proof}

From the equation $(\ref{eq: Lichnerowicz lemma})$, we immediately see that
\begin{equation} \label{eq: Lichnerowicz lemma2}
e_i \cdot e_j \cdot \bar{R}_S(e_i,e_j) = \frac{56}{3} - \frac{8}{3}\psi \cdot.
\end{equation}
Moreover, we also find that the curvature endomorphism $q_S(\bar{R})$ on the spinor bundle is
\begin{equation*}
q_S(\bar{R}) = \frac{14}{3} - \frac{2}{3}\psi \cdot.
\end{equation*}

\begin{theorem}
For the square of the twisted Dirac operator and the canonical $\G_2$-connection, we obtain that
\begin{equation} \label{eq: relation between twisted dirac and G2-Laplace}
\overline{D_{TM}}^2 = \bar{\Delta}_{S \otimes T} - \frac{2}{3} - \frac{2}{3}\psi \cdot \otimes \Id + \frac{2}{3}\big( (e_j \mathbin{\lrcorner} \phi) \cdot \otimes \Id \big) \bar{\nabla}_{e_j}.
\end{equation}
\end{theorem}
\begin{proof}
Calculating the same as the Lichnerowicz formula (cf. \cite[p.107]{Yasushi}), we get
\begin{equation*}
\overline{D_{TM}}^2 = \bar{\nabla}^{\ast}\bar{\nabla} + \frac{2}{3}\big( (e_j \mathbin{\lrcorner} \phi) \cdot \otimes \Id \big) \bar{\nabla}_{e_j} + \frac{1}{2}e_je_k \bar{R}_S(e_j,e_k) \otimes \Id + \frac{1}{2}e_je_k \cdot \otimes \bar{R}(e_j,e_k).
\end{equation*}
Applying the equation $(\ref{eq: Lichnerowicz lemma2})$, we see that the above equation becomes
\begin{equation} \label{eq: relation between twisted dirac and rough Laplacian}
\overline{D_{TM}}^2 = \bar{\nabla}^{\ast}\bar{\nabla} + \frac{28}{3} - \frac{4}{3}\psi\cdot \otimes \Id + \frac{2}{3}\big( (e_j \mathbin{\lrcorner} \phi) \cdot \otimes \Id \big) \bar{\nabla}_{e_j} + \frac{1}{2}e_je_k \cdot \otimes \bar{R}(e_j,e_k).
\end{equation}
As can be seen from the definition of the $\G_2$-Laplace operator $\bar{\Delta}_{S \otimes T}$, we only need to compute the curvature endomorphism $q_{S \otimes T}(\bar{R})$.
\begin{equation*}
 \begin{split}
 q_{S \otimes T}(\bar{R}) &= \frac{1}{2}\bar{R}(e_i \wedge e_j)_{\star} \otimes (e_i \wedge e_j)_{\star} + \frac{1}{2} (e_i \wedge e_j)_{\star} \otimes \bar{R}(e_i \wedge e_j)_{\star} \\
 &\quad + q_S(\bar{R}) \otimes \Id + \Id \otimes q_T(\bar{R}) \\
 &= 10 - \frac{2}{3}\psi \cdot \otimes \Id + \frac{1}{2}e_je_k \cdot \otimes \bar{R}(e_j,e_k).
 \end{split}
\end{equation*}
\end{proof}

The following lemma is obtained directly from the equations $(\ref{eq: G2 conn on tangent bundle})$, $(\ref{eq: G2 conn on spinor bundle})$, and the definition of the rough Laplacian $\nabla^{\ast}\nabla = - \nabla_{e_i}\nabla_{e_i} + \nabla_{\nabla_{e_i}e_i}$.
\begin{lemma}
On the sections of $S_{1/2} \otimes TM^{\C}$, we have
\begin{equation} \label{eq: rough Laplacian difference2}
\begin{split}
\bar{\nabla}^{\ast}\bar{\nabla} &= \nabla^{\ast}\nabla - \frac{5}{4} - \frac{1}{6}\psi \cdot \otimes \Id + \frac{1}{9}e_i \mathbin{\lrcorner} \phi \cdot \otimes A_{e_i} \\
&\quad + \frac{2}{3}(\Id \otimes A_{e_i})\bar{\nabla}_{e_i} + \frac{1}{3}(e_i \mathbin{\lrcorner} \phi \cdot \otimes \Id)\bar{\nabla}_{e_i}.
\end{split}
\end{equation}
\end{lemma}

Finally, we examine the difference in the square of the twisted Dirac operator for each connection.
\begin{theorem}
On the sections of $S_{1/2} \otimes TM^{\C}$, there is a relationship between two squared twisted Dirac operators ${D_{TM}}^2$ and $\overline{D_{TM}}^2$ such that
\begin{equation} \label{eq: square of twisted dirac difference}
\begin{split}
\overline{D_{TM}}^2 &= {D_{TM}}^2 + \frac{7}{12} - \frac{3}{2}\psi \cdot \otimes \Id + \frac{1}{9}e_j \mathbin{\lrcorner} \phi \cdot \otimes A_{e_j} + \frac{2}{3}(\Id \otimes A_{e_j})\bar{\nabla}_{e_j} \\
&\quad + (e_j \mathbin{\lrcorner} \phi \cdot \otimes \Id) \bar{\nabla}_{e_j} + \frac{1}{2}e_je_k \cdot \otimes \big( \bar{R}(e_j,e_k) - R(e_j,e_k) \big).
\end{split}
\end{equation}
Specifically, for a local section $\alpha^{(i)} \otimes e_i$ of $S_{1/2} \otimes TM^{\C}$, we have
\begin{equation*}
\begin{split}
\overline{D_{TM}}^2(\alpha^{(i)} \otimes e_i) &= {D_{TM}}^2(\alpha^{(i)} \otimes e_i) + \frac{13}{36}\alpha^{(i)} \otimes e_i - \frac{3}{2}\psi \cdot \alpha^{(i)} \otimes e_i \\
&\quad + \frac{5}{9} (e_i \mathbin{\lrcorner} e_j \mathbin{\lrcorner} \psi) \cdot \alpha^{(i)} \otimes e_j  - \frac{2}{9}e_je_i \cdot \alpha^{(i)} \otimes e_j \\
&\quad + \frac{2}{3}(\Id \otimes A_{e_j})\bar{\nabla}_{e_j}(\alpha^{(i)} \otimes e_i) + (e_j \mathbin{\lrcorner} \phi \cdot \otimes \Id) \bar{\nabla}_{e_j}(\alpha^{(i)} \otimes e_i).
\end{split}
\end{equation*}
\end{theorem}
\begin{proof}
Due to \cite[p.107]{Yasushi}, the Lichnerowicz formula for the twisted Dirac operator is 
\[
{D_{TM}}^2 = \nabla^{\ast}\nabla + \frac{15}{2} + \frac{1}{2}e_je_k\cdot \otimes R(e_j,e_k).
\]
Using the equation $(\ref{eq: relation between twisted dirac and rough Laplacian})$ and $(\ref{eq: rough Laplacian difference2})$, we obtain
\begin{equation*}
\begin{split}
\overline{D_{TM}}^2 &= \bar{\nabla}^{\ast}\bar{\nabla} + \frac{28}{3} - \frac{4}{3}\psi\cdot \otimes \Id + \frac{2}{3}\big( (e_j \mathbin{\lrcorner} \phi) \cdot \otimes \Id \big) \bar{\nabla}_{e_j} + \frac{1}{2}e_je_k \cdot \otimes \bar{R}(e_j,e_k) \\
&= \nabla^{\ast}\nabla + \frac{97}{12} - \frac{3}{2}\psi \cdot \otimes \Id + \frac{1}{9}(e_j \mathbin{\lrcorner} \phi) \cdot \otimes A_{e_j} + \frac{2}{3}(\Id \otimes A_{e_j})\bar{\nabla}_{e_j} \\
&\quad + \big((e_j \mathbin{\lrcorner} \phi) \cdot \otimes \Id \big) \bar{\nabla}_{e_j} + \frac{1}{2}e_je_k \cdot \otimes \bar{R}(e_j,e_k) \\
&= {D_{TM}}^2 + \frac{7}{12} - \frac{3}{2}\psi \cdot \otimes \Id + \frac{1}{9}(e_j \mathbin{\lrcorner} \phi) \cdot \otimes A_{e_j} + \frac{2}{3}(\Id \otimes A_{e_j})\bar{\nabla}_{e_j} \\
&\quad + \big((e_j \mathbin{\lrcorner} \phi) \cdot \otimes \Id \big) \bar{\nabla}_{e_j} + \frac{1}{2}e_je_k \cdot \otimes \big( \bar{R}(e_j,e_k) - R(e_j,e_k) \big).
\end{split}
\end{equation*}
We specifically compute certain two terms in the equation $(\ref{eq: square of twisted dirac difference})$ for a local section $\alpha^{(i)} \otimes e_i \in \Gamma(S_{1/2} \otimes TM^{\C})$. The first is
\begin{eqnarray*}
&&\frac{1}{9}(e_j \mathbin{\lrcorner} \phi) \cdot \alpha^{(i)} \otimes A_{e_j}e_i = \frac{1}{9}g(e_l,A_{e_j}e_i) (e_j \mathbin{\lrcorner} \phi) \cdot \alpha^{(i)} \otimes e_l \\
&=& - \frac{1}{9}\left( A_{e_l}e_i \mathbin{\lrcorner} \phi \right) \cdot \alpha^{(i)} \otimes e_l = - \frac{1}{9}\big( (e_i \mathbin{\lrcorner} e_l \mathbin{\lrcorner} \phi) \mathbin{\lrcorner} \phi \big) \cdot \alpha^{(i)} \otimes e_l \\
&\overset{(\ref{eq9})}{=}& \frac{1}{9}(e_i \mathbin{\lrcorner} e_l \mathbin{\lrcorner} \psi) \cdot \alpha^{(i)} \otimes e_l + \frac{1}{9}(e_i \wedge e_l) \cdot \alpha^{(i)} \otimes e_l \\
&=& \frac{1}{9}(e_i \mathbin{\lrcorner} e_j \mathbin{\lrcorner} \psi) \cdot \alpha^{(i)} \otimes e_j - \frac{1}{9}e_je_i \cdot \alpha^{(i)} \otimes e_j - \frac{1}{9}\alpha^{(i)} \otimes e_i.
\end{eqnarray*}
The second is the curvature term, and we get
\begin{eqnarray*}
&&\frac{1}{2}e_je_k \cdot \alpha^{(i)} \otimes \big( \bar{R}(e_j,e_k)e_i - R(e_j,e_k)e_i \big) \\
&\overset{{\rm Lem \;} \ref{lem: difference curvatures}}{=}&\frac{1}{2}e_je_k \cdot \alpha^{(i)} \otimes \frac{1}{9} \big( 4A_{A_{e_j}e_k}e_i - 3g(e_j,e_i)e_k + 3g(e_k,e_i)e_j \big) \\
&\overset{(\ref{eq1})}{=}& - \frac{2}{9}e_je_k \cdot \alpha^{(i)} \otimes \left( -g(e_i,e_j)e_k + g(e_i,e_k)e_j - \frac{1}{2}\chi(e_i,e_j,e_k) \right) \\
&& + \frac{1}{3}e_je_i \cdot \alpha^{(i)} \otimes e_j + \frac{1}{3}\alpha^{(i)} \otimes e_i \\
&=& -\frac{1}{9}e_j \cdot \chi(e_i,e_j,e_l) \cdot \alpha^{(i)} \otimes e_l - \frac{1}{9}e_je_i \cdot \alpha^{(i)} \otimes e_j - \frac{1}{9}\alpha^{(i)} \otimes e_i \\
&=& \frac{4}{9}(e_i \mathbin{\lrcorner} e_j \mathbin{\lrcorner} \psi) \cdot \alpha^{(i)} \otimes e_j - \frac{1}{9}e_je_i \cdot \alpha^{(i)} \otimes e_j - \frac{1}{9}\alpha^{(i)} \otimes e_i.
\end{eqnarray*}
\end{proof}

\section{Infinitesimal deformations of Killing spinors on nearly parallel $\G_2$-manifolds} \label{sec: infinitesimal deformation of Killing spinor}
Let $(M,\phi,g)$ be a compact nearly parallel $\G_2$-manifold with $\scal = 42$, and let $\kappa_0$ be a Killing spinor corresponding to the nearly parallel $\G_2$-structure $\phi$. The goal of this section is to identify the space of the infinitesimal deformations of the Killing spinor, and this is one of the main result in this paper.

 As we see in Subsection \ref{subsec: alg results on nearly G2}, we already know the decomposition of the spinor bundle $S_{1/2}$. Also, we have a $\G_2$-irreducible decomposition of $S_{1/2} \otimes TM^{\C}$:
\begin{equation} \label{eq: decomp of S otimes TM}
S_{1/2} \otimes TM^{\C} \cong \wedge^1M \oplus \wedge^2_7M \oplus \wedge^2_{14}M \oplus \Sym_0M \oplus \C g.
\end{equation}

Let us recall an infinitesimal deformation of the Killing spinor defined in Subsection \ref{subsec: infinitesimal deformations of killing spinors}. A pair $(H, \kappa)$ is an infinitesimal deformation of the Killing spinor $\kappa_0$ with constant $c$ if the symmetric endomorphism $H:TM \to TM$ and the spinor $\kappa$ satisfy the following conditions:
 \begin{enumerate}
 \renewcommand{\labelenumi}{(\roman{enumi})}
 \item $\kappa$ is a Killing spinor with constant $c$. 
 \item $\tr H = \delta H=0$.
 \item $D_{TM} \Psi^{(H, \kappa_0)} = nc\Psi^{(H, \kappa_0)}$.
 \end{enumerate}
Note that we deal with the case of the dimension $n=7$ and the constant $c=\frac{1}{2}$. The symmetric endomorphism $H$ is represented locally as ${\alpha_1}^{(i)} \odot e_i$ using a local orthonormal frame $\{ e_i \}$ which is $\bar{\nabla}$-parallel at a point $x \in M$ for simplicity, and vector fields ${\alpha_1}^{(i)}$. Here, the symbol $\odot$ is the symmetric tensor product defined by ${\alpha_1}^{(i)} \odot e_i \coloneqq {\alpha_1}^{(i)} \otimes e_i + e_i \otimes {\alpha_1}^{(i)}$. With this local representation, the condition (iii) becomes
\begin{equation} \label{D_{TM}=7/2}
D_{TM}\left( {\alpha_1}^{(i)} \cdot \kappa_0 \otimes e_i + e_i \cdot \kappa_0 \otimes {\alpha_1}^{(i)} \right) = \frac{7}{2}\left( {\alpha_1}^{(i)} \cdot \kappa_0 \otimes e_i + e_i \cdot \kappa_0 \otimes {\alpha_1}^{(i)} \right).
\end{equation}
Using the equation $(\ref{eq: twisted dirac relation})$, we rewrite the equation $(\ref{D_{TM}=7/2})$ as
\begin{equation} \label{D_{TM}=7/2'}
\begin{split}
&\overline{D_{TM}}\left( {\alpha_1}^{(i)} \cdot \kappa_0 \otimes e_i + e_i \cdot \kappa_0 \otimes {\alpha_1}^{(i)} \right) + \frac{1}{2}\phi \cdot {\alpha_1}^{(i)} \cdot \kappa_0 \otimes e_i + \frac{1}{2}\phi \cdot e_i \cdot \kappa_0 \otimes {\alpha_1}^{(i)} \\
&+ \frac{1}{3}e_j \cdot {\alpha_1}^{(i)} \cdot \kappa_0 \otimes A_{e_j}e_i + \frac{1}{3}e_j \cdot e_i \cdot \kappa_0 \otimes A_{e_j}{\alpha_1}^{(i)} \\
&= \frac{7}{2}\left( {\alpha_1}^{(i)} \cdot \kappa_0 \otimes e_i + e_i \cdot \kappa_0 \otimes {\alpha_1}^{(i)} \right).
\end{split}
\end{equation}
Using the formula $(\ref{eq: clifford action})$, we compute each term in the left-hand side of the equation $(\ref{D_{TM}=7/2'})$. For the first term of the equation $(\ref{D_{TM}=7/2'})$, we have
\begin{align*}
\overline{D_{TM}}({\alpha_1}^{(i)} \cdot \kappa_0 \otimes e_i) &= - \bar{\nabla}_{e_j} \big( g(e_j,{\alpha_1}^{(i)})\kappa_0 \otimes e_i \big) + \bar{\nabla}_{e_j} \big(A_{e_j}{\alpha_1}^{(i)} \cdot \kappa_0 \otimes e_i \big) \\
&\quad + g(\bar{\nabla}_{e_j}e_j, {\alpha_1}^{(i)})\kappa_0 \otimes e_i - A_{\bar{\nabla}_{e_j}e_j}{\alpha_1}^{(i)} \cdot \kappa_0 \otimes e_i, \\
\overline{D_{TM}}(e_i \cdot \kappa_0 \otimes {\alpha_1}^{(i)}) &= - \bar{\nabla}_{e_j} \big( g(e_j,e_i)\kappa_0 \otimes {\alpha_1}^{(i)} \big) + \bar{\nabla}_{e_j} \big(A_{e_j}e_i \cdot \kappa_0 \otimes {\alpha_1}^{(i)} \big) \\
&\quad + g(\bar{\nabla}_{e_j}e_j, e_i)\kappa_0 \otimes {\alpha_1}^{(i)} - A_{\bar{\nabla}_{e_j}e_j}e_i \cdot \kappa_0 \otimes {\alpha_1}^{(i)}.
\end{align*}
Similarly, for the second and third term of the equation $(\ref{D_{TM}=7/2'})$, we have
\begin{align*}
\phi \cdot {\alpha_1}^{(i)} \cdot \kappa_0 \otimes e_i &= {\alpha_1}^{(i)} \cdot \kappa_0 \otimes e_i, \\
\phi \cdot e_i \cdot \kappa_0 \otimes {\alpha_1}^{(i)} &= e_i \cdot \kappa_0 \otimes {\alpha_1}^{(i)}.
\end{align*}
Finally, for the fourth and fifth term of the equation $(\ref{D_{TM}=7/2'})$, we have
\begin{align*}
e_j \cdot {\alpha_1}^{(i)} \cdot \kappa_0 \otimes A_{e_j}e_i &= \kappa_0 \otimes A_{e_i}{\alpha_1}^{(i)} + A_{e_j}{\alpha_1}^{(i)} \cdot \kappa_0 \otimes A_{e_j}e_i, \\
e_j \cdot e_i \cdot \kappa_0 \otimes A_{e_j}{\alpha_1}^{(i)} &= \kappa_0 \otimes A_{{\alpha_1}^{(i)}}e_i + A_{e_j}e_i \cdot \kappa_0 \otimes A_{e_j}{\alpha_1}^{(i)}.
\end{align*}
We note that 
\begin{equation*}
\begin{split}
\delta ({\alpha_1}^{(i)} \odot e_i) &\coloneqq -g(e_j,\bar{\nabla}_{e_j}{\alpha_1}^{(i)})e_i - g(e_j,{\alpha_1}^{(i)})\bar{\nabla}_{e_j}e_i \\
&\quad -g(e_j,e_i)\bar{\nabla}_{e_j}{\alpha_1}^{(i)} - g(e_j,\bar{\nabla}_{e_j}e_i){\alpha_1}^{(i)} \\
&= -\bar{\nabla}_{e_j}\left( g(e_j,{\alpha_1}^{(i)})e_i \right) - \bar{\nabla}_{e_j}\left( g(e_j,e_i){\alpha_1}^{(i)} \right) \quad ({\rm at \;} x)
\end{split}
\end{equation*}
holds and summarize the above calculations, then (at $x$) the equation $(\ref{D_{TM}=7/2'})$ becomes
\begin{equation} \label{D_{TM}=7/2''}
\begin{split}
&\kappa_0 \otimes \delta({\alpha_1}^{(i)} \odot e_i) + \bar{\nabla}_{e_j} \big(A_{e_j}{\alpha_1}^{(i)} \cdot \kappa_0 \otimes e_i \big) + \bar{\nabla}_{e_j} \big(A_{e_j}e_i \cdot \kappa_0 \otimes {\alpha_1}^{(i)} \big) \\
&+ \frac{1}{3}A_{e_j}{\alpha_1}^{(i)} \cdot \kappa_0 \otimes A_{e_j}e_i + \frac{1}{3}A_{e_j}e_i \cdot \kappa_0 \otimes A_{e_j}{\alpha_1}^{(i)} \\
&=3\left( {\alpha_1}^{(i)} \cdot \kappa_0 \otimes e_i + e_i \cdot \kappa_0 \otimes {\alpha_1}^{(i)} \right).
\end{split}
\end{equation}
We project the equation $(\ref{D_{TM}=7/2''})$ onto the irreducible component $\Sym_0M$ along the decomposition $(\ref{eq: decomp of S otimes TM})$. Then we get
\begin{equation} \label{D_{TM}=7/2, Sym_0}
(A_{e_j} \otimes \Id)\bar{\nabla}_{e_j}\left( {\alpha_1}^{(i)} \odot e_i \right) + \frac{1}{3}\left( A_{e_j}{\alpha_1}^{(i)} \otimes A_{e_j}e_i + A_{e_j}e_i \otimes A_{e_j}{\alpha_1}^{(i)} \right) = 3\left( {\alpha_1}^{(i)} \odot e_i \right).
\end{equation}
Let us calculate the second term on the left side of the equation $(\ref{D_{TM}=7/2, Sym_0})$. 
\begin{eqnarray*}
A_{e_j}{\alpha_1}^{(i)} \otimes A_{e_j}e_i &=& e_j \otimes A_{e_i}A_{e_j}{\alpha_1}^{(i)} \\
&\overset{(\ref{eq1})}{=}& e_j \otimes \left( -g(e_i,e_j){\alpha_1}^{(i)} + g(e_i,{\alpha_1}^{(i)})e_j - \frac{1}{2}\chi(e_i,e_j,{\alpha_1}^{(i)}) \right) \\
&=& -e_i \otimes {\alpha_1}^{(i)} - \frac{1}{4}e_j \odot \chi(e_i,e_j,{\alpha_1}^{(i)}) - \frac{1}{4}e_j \wedge \chi(e_i,e_j,{\alpha_1}^{(i)}) \\
&=& -e_i \otimes {\alpha_1}^{(i)} + e_i \mathbin{\lrcorner} {\alpha_1}^{(i)} \mathbin{\lrcorner} \psi \overset{(\ref{eq12})}{=} -e_i \otimes {\alpha_1}^{(i)} + \ast(\phi \wedge {\alpha_1}^{(i)} \wedge e_i), \\
A_{e_j}e_i \otimes A_{e_j}{\alpha_1}^{(i)} &=& -{\alpha_1}^{(i)} \otimes e_i + \ast(\phi \wedge e_i \wedge {\alpha_1}^{(i)}).
\end{eqnarray*}
Furthermore, using the two actions $A_{e_j \star}$ and $\widetilde{A_{e_j \star}}$ on 2-tensors introduced in Subsection \ref{subsec: action of cross product}, we find that the equation $(\ref{D_{TM}=7/2, Sym_0})$ becomes
\begin{equation} \label{D_{TM}=7/2, Sym_0'}
\frac{1}{2}A_{e_j \star}\bar{\nabla}_{e_j}H + \frac{1}{2}\widetilde{A_{e_j \star}}\bar{\nabla}_{e_j}H = \frac{10}{3}H.
\end{equation}
Let $\gamma$ be a section of $\wedge^3_{27}M$ defined by $\gamma \coloneqq \mathbf{i}(H)$. Applying the result of Proposition \ref{prop: cross product action2} to the equation $(\ref{D_{TM}=7/2, Sym_0'})$, we get
\begin{equation} \label{D_{TM}=7/2, Sym_0''}
A_{e_j \star}\bar{\nabla}_{e_j}H - \frac{1}{3}\delta H \mathbin{\lrcorner} \phi + \frac{1}{2}\left( \delta\gamma \right)_{\wedge^2_{14}} = \frac{20}{3}H.
\end{equation}
The equation $(\ref{D_{TM}=7/2, Sym_0''})$ is equivalent to the system
\begin{equation} \label{system a}
\begin{cases}
- \frac{1}{3}\delta H \mathbin{\lrcorner} \phi + \frac{1}{2}\left( \delta\gamma \right)_{\wedge^2_{14}} = 0, \\
A_{e_j \star}\bar{\nabla}_{e_j}H = \frac{20}{3}H.
\end{cases}
\end{equation}
Since $\left( \delta \gamma \right)_{\wedge^2_7} = \frac{4}{3}\delta H \mathbin{\lrcorner} \phi$ holds, $\delta H$ vanishes if and only if $\delta\gamma$ is in $\Omega^2_{14}M$ (i.e. $\delta\gamma = \left( \delta\gamma \right)_{\wedge^2_{14}}$). Thus, the first equation of $(\ref{system a})$ becomes $\delta \gamma =0$. As for the second equation of $(\ref{system a})$, sending both sides to $\wedge^3_{27}M$ by the isomorphism $\mathbf{i}$ and employing the result of Proposition \ref{prop: cross product action1}, we find
\begin{equation} \label{eq: main infinitesimal deformation of Killing spinor}
\ast d\gamma = -4\gamma,
\end{equation}
where we use the fact that $\delta H=0$ is equivalent to $\ast d\gamma = \left( \ast d\gamma \right)_{\wedge^3_{27}}$. The section $\gamma \in \Omega^3_{27}M$ satisfying the equation $(\ref{eq: main infinitesimal deformation of Killing spinor})$ is automatically co-closed. We obtain the main result.

\let\temp\thetheorem
\renewcommand{\thetheorem}{\ref{thmA}}

\begin{theorem}
Let $(M,\phi,g)$ be a compact nearly parallel $\G_2$-manifold and let $\kappa_0$ be a Killing spinor with constant $\frac{1}{2}$ corresponding to the nearly parallel $\G_2$-structure $\phi$. The space of infinitesimal deformations of the Killing spinor $\kappa_0$ is isomorphic to the direct sum of the space $D_3 \coloneqq \left\{ \gamma \in \Omega^3_{27}M \mid \ast d\gamma = -4\gamma \right\}$ and the space $K_+$ of all Killing spinors.
\end{theorem}

\let\thetheorem\temp
\addtocounter{theorem}{-1}

It is proven in \cite[Theorem 3.5]{AlexandrovSemmelmann} that the space of the infinitesimal deformations of a nearly parallel $\G_2$-structure $(\phi,g)$ is given by the direct sum of the spaces 
\[
D_3 \quad {\rm and} \quad D_1 \coloneqq \left\{ f_1 \in \Omega^1M \mid \nabla f_1 = -f_1 \mathbin{\lrcorner} \phi \right\}.
\]
Furthermore, the dimension of this space $D_1$ is specified in \cite[Theorem 4.2]{AlexandrovSemmelmann} as follows.
\begin{center}
$(M,\phi,g)$ is of type 1 if and only if $\dim D_1=0$, \\
$(M,\phi,g)$ is of type 2 if and only if $\dim D_1=1$, \\
$(M,\phi,g)$ is of type 3 if and only if $\dim D_1=2$.
\end{center}
We clarify the relationship between the spaces $D_1$ and $K_+$. An one-dimensional vector space spanned by the Killing spinor $\kappa_0$ is represented by $\langle \kappa_0 \rangle$, and the orthogonal complement of $\langle \kappa_0 \rangle$ in the space $K_+$ is represented by $\langle \kappa_0 \rangle^{\perp}$. Then, the map $D_1 \ni f_1 \mapsto f_1 \cdot \kappa_0 \in \langle \kappa_0 \rangle^{\perp}$ gives an isomorphism between the spaces $D_1$ and $\langle \kappa_0 \rangle^{\perp}$. Also, the deformation to the direction of $\kappa_0$ is a trivial deformation of the Killing spinor. Thus, Theorem \ref{thmA} is a reproof of the result of \cite[Theorem 3.5]{AlexandrovSemmelmann} through the Killing spinor.

\section{Rarita-Schwinger fields on nearly parallel $\G_2$-manifolds} \label{sec: Rarita-Schwinger fields}
We define Rarita-Schwinger fields using the twisted Dirac operator defined in Subsection \ref{subsec: infinitesimal deformations of killing spinors}. {\it A Rarita-Schwinger field} is a section $\varphi$ of $S_{1/2} \otimes TM^{\C}$ that satisfies
\[
\varphi \in \Gamma(S_{3/2}) \quad {\rm and} \quad D_{TM}\varphi=0.
\]
The properties and applications of Rarita-Schwinger fields are listed in \cite{BarBandara}, \cite{BarMazzeo}, \cite{YasushiSemmelmann}, and \cite{OhnoTomihisa}, for example.

We can identify the space of the Rarita-Schwinger fields through a similar technique as when we investigate infinitesimal deformations of Killing spinors in Section \ref{sec: infinitesimal deformation of Killing spinor}. This is the other main theorem in this paper.

We denote the corresponding Killing spinor to the nearly parallel $\G_2$-structure $(\phi,g)$ by $\kappa_0$. Any element of $\Gamma(S_{1/2} \otimes TM^{\C})$ is expressed locally as $\alpha^{(i)} \otimes e_i$ using a local orthonormal frame $\{ e_i \}$ and spinor fields $\alpha^{(i)}$. This $\alpha^{(i)} \in \Gamma(S_{1/2})$ is decomposed to 
\begin{equation} \label{alpha decomp}
\alpha^{(i)} = \left( {\alpha_0}^{(i)} + {\alpha_1}^{(i)} \right) \cdot \kappa_0 \in \left( \Omega^0M \oplus \Omega^1M \right) \cdot \kappa_0.
\end{equation} 

\begin{lemma}
Let $\alpha^{(i)} \otimes e_i$ be in $S_{1/2} \otimes TM^{\C}$. Then $\alpha^{(i)} \otimes e_i$ is in $S_{3/2}$ if and only if
\begin{numcases}
{}
\tr({\alpha_1}^{(i)} \otimes e_i)=0, & \label{eq: S_{3/2} condition omega^0} \\ 
{\alpha_0}^{(i)} e_i + A_{e_i}{\alpha_1}^{(i)}=0. \label{eq: S_{3/2} condition omega^1}
\end{numcases}
\end{lemma}
\begin{proof}
By definition, $\alpha^{(i)} \otimes e_i$ is in $S_{3/2}$ if and only if $e_i \cdot \alpha^{(i)}$ vanishes. On the other hand, using the equation $(\ref{eq: clifford action})$ we get
\[
e_i \cdot \alpha^{(i)} = -g(e_i,{\alpha_1}^{(i)})\kappa_0 + \left({\alpha_0}^{(i)}e_i + A_{e_i}{\alpha_1}^{(i)} \right) \cdot \kappa_0.
\]
Since $g(e_i,{\alpha_1}^{(i)})$ is a trace part of $\alpha^{(i)} \otimes e_i$, we get the equation $(\ref{eq: S_{3/2} condition omega^0})$ and finish the proof.
\end{proof}

\begin{lemma} \label{lemma: alpha_0}
We denote a Rarita-Schwinger field $\varphi$ locally as $\alpha^{(i)} \otimes e_i$, then we obtain ${\alpha_0}^{(i)}e_i=0$.
\end{lemma}
\begin{proof}
We take a local orthonormal frame $\{ e_i \}$ which is $\bar{\nabla}$-parallel at a point $x$. Using the equation $(\ref{eq: twisted dirac relation})$, we rewrite the equation $D_{TM}\left( \alpha^{(i)} \otimes e_i \right) = 0$ as
\begin{equation} \label{D_{TM}=0}
\overline{D_{TM}}\left(\alpha^{(i)} \otimes e_i \right) + \frac{1}{2}\phi \cdot \alpha^{(i)} \otimes e_i + \frac{1}{3} e_j \cdot \alpha^{(i)} \otimes A_{e_j}e_i = 0.
\end{equation}
By diverting the method of computing the equation $(\ref{D_{TM}=7/2'})$, we see that (at $x$) the equation $(\ref{D_{TM}=0})$ becomes
\begin{equation} \label{eq: Rarita-Schwinger field condition important}
\begin{split}
0 &= \kappa_0 \otimes \left( -\frac{7}{2}{\alpha_0}^{(i)}e_i + \frac{1}{3}A_{e_i}{\alpha_1}^{(i)} - \bar{\nabla}_{e_j}\big( g(e_j, {\alpha_1}^{(i)})e_i \big) \right) \\
&\quad + \bar{\nabla}_{e_j} \big(A_{e_j}{\alpha_1}^{(i)} \cdot \kappa_0 \otimes e_i \big) + e_j \cdot \kappa_0 \otimes \bar{\nabla}_{e_j}({\alpha_0}^{(i)}e_i) + \frac{1}{2}{\alpha_1}^{(i)} \cdot \sigma \otimes e_i \\
&\quad + \frac{1}{3}e_j \cdot \sigma \otimes {\alpha_0}^{(i)}A_{e_j}e_i + \frac{1}{3}A_{e_j}{\alpha_1}^{(i)} \cdot \sigma \otimes A_{e_j}e_i.
\end{split}
\end{equation}
We project the equation $(\ref{eq: Rarita-Schwinger field condition important})$ onto each space of the irreducible decomposition $(\ref{eq: decomp of S otimes TM})$. First, projecting the equation $(\ref{eq: Rarita-Schwinger field condition important})$ onto the bundle $\wedge^0M \otimes \wedge^1M \cong \wedge^1M$, we get
\begin{eqnarray}
0 &=& -\frac{7}{2}{\alpha_0}^{(i)}e_i + \frac{1}{3}A_{e_i}{\alpha_1}^{(i)} - \bar{\nabla}_{e_j}\big( g(e_j, {\alpha_1}^{(i)})e_i \big) \nonumber \\
&\overset{(\ref{eq: S_{3/2} condition omega^1})}{=}& - \frac{23}{6}{\alpha_0}^{(i)}e_i - \bar{\nabla}_{e_j}\big( g(e_j, {\alpha_1}^{(i)})e_i \big) \quad ({\rm at \;} x). \label{eq: Rarita-Schwinger field condition important proj to wedge^0}
\end{eqnarray}
We remark that
\begin{equation*}
\begin{split}
-\frac{1}{2}\delta({\alpha_1}^{(i)} \wedge e_i) - \frac{1}{2}\delta({\alpha_1}^{(i)} \odot e_i) &= \frac{1}{2}g(e_j,\bar{\nabla}_{e_j}{\alpha_1}^{(i)})e_i - \frac{1}{2}g(e_j,e_i)\bar{\nabla}_{e_j}{\alpha_1}^{(i)} \\
&\quad - \frac{1}{2}g(e_j,{\alpha_1}^{(i)})\bar{\nabla}_{e_j}e_i + \frac{1}{2}g(e_j,\bar{\nabla}_{e_j}e_i){\alpha_1}^{(i)} \\
&\quad + \frac{1}{2}g(e_j,\bar{\nabla}_{e_j}{\alpha_1}^{(i)})e_i + \frac{1}{2}g(e_j,e_i)\bar{\nabla}_{e_j}{\alpha_1}^{(i)} \\
&\quad + \frac{1}{2}g(e_j,{\alpha_1}^{(i)})\bar{\nabla}_{e_j}e_i + \frac{1}{2}g(e_j,\bar{\nabla}_{e_j}e_i){\alpha_1}^{(i)} \\
&= \bar{\nabla}_{e_j}\big( g(e_j, {\alpha_1}^{(i)})e_i \big) \quad ({\rm at \;} x).
\end{split}
\end{equation*}

Now, we consider another equation ${D_{TM}}^2 \left( \alpha^{(i)} \otimes e_i \right)=0$. According to the equation $(\ref{eq: relation between twisted dirac and G2-Laplace})$ and $(\ref{eq: square of twisted dirac difference})$, ${D_{TM}}^2 \left( \alpha^{(i)} \otimes e_i \right)=0$ is equivalent to 
\begin{equation} \label{eq: {D_{TM}}^2=0 iff}
\begin{split}
\bar{\Delta}_{S \otimes T}(\alpha^{(i)} \otimes e_i) &= \frac{37}{36}\alpha^{(i)} \otimes e_i - \frac{5}{6} \psi \cdot \alpha^{(i)} \otimes e_i + \frac{5}{9}(e_i \mathbin{\lrcorner} e_j \mathbin{\lrcorner} \psi) \cdot \alpha^{(i)} \otimes e_j \\
&\quad - \frac{2}{9}e_je_i \cdot \alpha^{(i)} \otimes e_j + \frac{2}{3} (\Id \otimes A_{e_j})\bar{\nabla}_{e_j}(\alpha^{(i)} \otimes e_i) \\
&\quad + \frac{1}{3}(e_j \mathbin{\lrcorner} \phi \cdot \otimes \Id)\bar{\nabla}_{e_j}(\alpha^{(i)} \otimes e_i).
\end{split}
\end{equation}
Using the decomposition $(\ref{alpha decomp})$ and the formula $(\ref{eq: clifford action})$, we compute each term in the right-hand side of the equation $(\ref{eq: {D_{TM}}^2=0 iff})$.
\begin{align*}
\frac{37}{36}\alpha^{(i)} \otimes e_i &= \frac{37}{36}\kappa_0 \otimes {\alpha_0}^{(i)} e_i + \frac{37}{36}{\alpha_1}^{(i)} \cdot \kappa_0 \otimes e_i, \\
-\frac{5}{6}\psi \cdot \alpha^{(i)} \otimes e_i &= -\frac{35}{6}\kappa_0 \otimes {\alpha_0}^{(i)}e_i + \frac{5}{6}{\alpha_1}^{(i)} \cdot \kappa_0 \otimes e_i, \\
\frac{5}{9}(e_i \mathbin{\lrcorner} e_j \mathbin{\lrcorner} \psi) \cdot \alpha^{(i)} \otimes e_j &= \frac{10}{9}\kappa_0 \otimes A_{e_i}{\alpha_1}^{(i)} - \frac{10}{9}A_{e_j}{\alpha_0}^{(i)}e_i \cdot \kappa_0 \otimes e_j \\
&\quad - \frac{10}{9}A_{{\alpha_1}^{(i)}}A_{e_j}e_i \cdot \kappa_0 \otimes e_j - \frac{5}{9}\chi(e_j,e_i,{\alpha_1}^{(i)}) \cdot \kappa_0 \otimes e_j, \\
-\frac{2}{9}e_je_i \cdot \alpha^{(i)} \otimes e_j &= \frac{2}{9}\kappa_0 \otimes {\alpha_0}^{(i)}e_i + \frac{2}{9}\kappa_0 \otimes A_{e_i}{\alpha_1}^{(i)} \\
&\quad - \frac{2}{9}A_{e_j}{\alpha_0}^{(i)}e_i \cdot \kappa_0 \otimes e_j + \frac{2}{9}g(e_i, {\alpha_1}^{(i)})e_j \cdot \kappa_0 \otimes e_j \\
&\quad - \frac{2}{9}A_{e_j}A_{e_i}{\alpha_1}^{(i)} \cdot \kappa_0 \otimes e_j, 
\\
\frac{2}{3} (\Id \otimes A_{e_j})\bar{\nabla}_{e_j}(\alpha^{(i)} \otimes e_i) &= \frac{2}{3}\kappa_0 \otimes A_{e_j} \bar{\nabla}_{e_j}({\alpha_0}^{(i)}e_i) + \frac{2}{3}(\Id \otimes A_{e_j})\bar{\nabla}_{e_j}({\alpha_1}^{(i)}\cdot \kappa_0 \otimes e_i), 
\end{align*}
\begin{align*}
\frac{1}{3}(e_j \mathbin{\lrcorner} \phi \cdot \otimes \Id)\bar{\nabla}_{e_j}(\alpha^{(i)} \otimes e_i) &= -\kappa_0 \otimes \bar{\nabla}_{e_j}(g(e_j,{\alpha_1}^{(i)})e_i) + g(\bar{\nabla}_{e_j}e_j, {\alpha_1}^{(i)})\kappa_0 \otimes e_i \\
&\quad - \frac{1}{3}\bar{\nabla}_{e_j}(A_{e_j}{\alpha_1}^{(i)} \cdot \kappa_0 \otimes e_i) + \frac{1}{3}A_{\bar{\nabla}_{e_j}e_j}{\alpha_1}^{(i)} \cdot \kappa_0 \otimes e_i \\
&\quad + e_j \cdot \kappa_0 \otimes \bar{\nabla}_{e_j}({\alpha_0}^{(i)}e_i).
\end{align*}
Taking into account the results of these calculations, we project the equation $(\ref{eq: {D_{TM}}^2=0 iff})$ onto the bundle $\wedge^0M \otimes \wedge^1M \cong \wedge^1M$. We then obtain
\begin{eqnarray*}
\bar{\Delta}({\alpha_0}^{(i)}e_i) &=& -\frac{55}{12}{\alpha_0}^{(i)}e_i + \frac{4}{3}A_{e_i}{\alpha_1}^{(i)} - \bar{\nabla}_{e_j}(g(e_j,{\alpha_1}^{(i)})e_i) + \frac{2}{3}A_{e_j}\bar{\nabla}_{e_j}({\alpha_0}^{(i)}e_i) \\
&=& - \frac{25}{12}{\alpha_0}^{(i)}e_i + \frac{2}{3}A_{e_j}\bar{\nabla}_{e_j}({\alpha_0}^{(i)}e_i).
\end{eqnarray*}
Since $\bar{\Delta} - \Delta = \frac{2}{3}A_{e_j}\bar{\nabla}_{e_j} - \frac{4}{3}$ holds over $\wedge^1M$, the above equation becomes
\[
\Delta({\alpha_0}^{(i)}e_i) = - \frac{3}{4}{\alpha_0}^{(i)}e_i.
\]
As you know, the Laplace operator $\Delta$ on a tangent bundle is a non-negative operator, so we obtain ${\alpha_0}^{(i)}e_i=0$.
\end{proof}

\let\temp\thetheorem
\renewcommand{\thetheorem}{\ref{thmB}}

\begin{theorem}
Let $(M,\phi,g)$ be a compact nearly parallel $\G_2$-manifold with scalar curvature $\scal=42$. Then the space of the Rarita-Schwinger fields is isomorphic to
\[
R_{\gamma} \coloneqq \left\{ \gamma \in \Omega^3_{27}M \; \middle| \; \ast d\gamma = -\frac{1}{2}\gamma \right\}.
\]
Clearly, this space $R_{\gamma}$ is contained in the $1/4$-eigenspace $\left\{ \gamma \in \Omega^3_{27}M \; \middle| \; \Delta\gamma = \frac{1}{4}\gamma \right\}$ of the Laplace operator.
\end{theorem}

\let\thetheorem\temp
\addtocounter{theorem}{-1}

\begin{proof}
Before beginning the proof, we provide two symbols for a 2-tensor field ${\alpha_1}^{(i)} \otimes e_i$:
\begin{eqnarray*}
w &\coloneqq& \pr_{\wedge^2M}({\alpha_1}^{(i)} \otimes e_i) = \frac{1}{2} {\alpha_1}^{(i)} \wedge e_i = \frac{1}{2}({\alpha_1}^{(i)} \otimes e_i - e_i \otimes {\alpha_1}^{(i)}), \\
H &\coloneqq& \pr_{\Sym M}({\alpha_1}^{(i)} \otimes e_i) = \frac{1}{2} {\alpha_1}^{(i)} \odot e_i = \frac{1}{2} ({\alpha_1}^{(i)} \otimes e_i + e_i \otimes {\alpha_1}^{(i)}).
\end{eqnarray*}
Let $\alpha^{(i)} \otimes e_i$ be a local expression of a Rarita-Schwinger field. Applying Lemma \ref{lemma: alpha_0} to the equation $(\ref{eq: S_{3/2} condition omega^1})$ yields $A_{e_i}{\alpha_1}^{(i)}=0$. This implies that ${\alpha_1}^{(i)} \wedge e_i$ is in $\Omega^2_{14}M$. Furthermore, applying Lemma \ref{lemma: alpha_0} to the equation $(\ref{eq: Rarita-Schwinger field condition important proj to wedge^0})$, we get $\bar{\nabla}_{e_j}\big( g(e_j, {\alpha_1}^{(i)})e_i \big)=0$, which means $\delta w + \delta H=0$.

Now, applying Lemma \ref{lemma: alpha_0} to the key equation $(\ref{eq: Rarita-Schwinger field condition important})$ and projecting it onto the bundle $\wedge^1M \otimes \wedge^1M$, we obtain
\begin{equation} \label{eq: proj to wedge^1 otimes wedge^1 1}
(A_{e_j} \otimes \Id) \bar{\nabla}_{e_j}({\alpha_1}^{(i)} \otimes e_i) = -\frac{1}{2}{\alpha_1}^{(i)} \otimes e_i - \frac{1}{3}A_{e_j}{\alpha_1}^{(i)} \otimes A_{e_j}e_i.
\end{equation}
As we obtained when calculating the equation $(\ref{D_{TM}=7/2, Sym_0})$, $A_{e_j}{\alpha_1}^{(i)} \otimes A_{e_j}e_i = -e_i \otimes {\alpha_1}^{(i)} + \ast(\phi \wedge {\alpha_1}^{(i)} \wedge e_i)$ holds. From the characterization of the bundle $\wedge^2_{14}M$, it follows that $\ast(\phi \wedge {\alpha_1}^{(i)} \wedge e_i) = {\alpha_1}^{(i)} \wedge e_i$. Therefore, the equation $(\ref{eq: proj to wedge^1 otimes wedge^1 1})$ is rewritten as
\begin{equation} \label{eq: proj to wedge^1 otimes wedge^1 3}
\frac{1}{2}A_{e_j \star}\bar{\nabla}_{e_j}({\alpha_1}^{(i)} \otimes e_i) + \frac{1}{2}\widetilde{A_{e_j \star}}\bar{\nabla}_{e_j}({\alpha_1}^{(i)} \otimes e_i) = -\frac{3}{2}w - \frac{1}{6}H.
\end{equation}
Employing the result of Proposition \ref{prop: cross product action1} and Proposition \ref{prop: cross product action2} to the equation $(\ref{eq: proj to wedge^1 otimes wedge^1 3})$, we get
\begin{equation}
\begin{split} \label{eq: proj to wedge^1 otimes wedge^1 4}
\delta w \mathbin{\lrcorner} \phi - \frac{1}{3}\delta H \mathbin{\lrcorner} \phi + \frac{1}{2}(\delta \gamma)_{\wedge^2_{14}} + \widetilde{A_{e_j \star}}\bar{\nabla}_{e_j}w + A_{e_j \star}\bar{\nabla}_{e_j}H = -3w -\frac{1}{3}H, 
\end{split}
\end{equation}
where $\gamma$ is a section of $\wedge^3_{27}M$ defined by $\gamma = \mathbf{i}(H)$. In the same way that we show Theorem \ref{thmA}, the equation $(\ref{eq: proj to wedge^1 otimes wedge^1 4})$ is equivalent to the system
\begin{equation*}
\gamma \in \Omega^3_{27}M, \quad 4d\delta\gamma + 6\ast d\gamma + 3\gamma = 0, \quad (\delta\gamma)_{\wedge^2_7}=0, \quad w=-\frac{1}{6}\delta\gamma.
\end{equation*}
From the above calculations, the space of Rarita-Schwinger fields is isomorphic to 
\begin{equation} \label{space isomorphic to RS-fields}
\left\{ \gamma \in \Omega^3_{27}M \; \middle| \; 4d\delta\gamma + 6\ast d\gamma + 3\gamma = 0, (\delta\gamma)_{\wedge^2_7}=0 \right\}.
\end{equation}
This space is included in 
\begin{equation} \label{space includes RS-fields}
\left\{ \gamma \in \Omega^3_{27}M \; \middle| \; 4\Delta\gamma + 8\ast d\gamma + 3\gamma = 0, (\delta\gamma)_{\wedge^2_7}=0 \right\}.
\end{equation}
We first examine the large space $(\ref{space includes RS-fields})$. The method we describe hereafter is the same as in \cite[Theorem 6.2]{AlexandrovSemmelmann}. Since the space $(\ref{space includes RS-fields})$ is a subspace of the kernel of a second-order elliptic differential operator $4\Delta + 8\ast d + 3\Id$ on $\Omega^3M$, its dimension is finite. The operator $\ast d$ is symmetric in the space $\ker (4\Delta + 8\ast d + 3\Id) \subset \Omega^3M$ and preserves the condition $(\delta\gamma)_{\wedge^2_7}=0$. As mentioned in Section \ref{sec: infinitesimal deformation of Killing spinor}, we note that from $(\delta\gamma)_{\wedge^2_7}=0$, the operator $\ast d$ satisfies $\ast d\gamma = (\ast d\gamma)_{\wedge^3_{27}}$. Thus, we can decompose the space $(\ref{space includes RS-fields})$ into the eigenspaces of the operator $\ast d$. If we assume $\ast d \gamma = \lambda \gamma \; (\lambda \neq 0)$, then $\gamma$ is co-closed and the quadratic equation $4\lambda^2 + 8\lambda +3=0$ is derived, whose solutions are $\lambda = -\frac{1}{2},-\frac{3}{2}$. In the case $\lambda = 0$, we get that $\gamma$ is closed and $d\delta\gamma = -\frac{3}{4}\gamma$. In summary, we have that the space $(\ref{space includes RS-fields})$ is isomorphic to the direct sum of the spaces
\begin{equation} \label{RS-field direct sum}
\left\{ \gamma \in \Omega^3_{27}M \middle| \ast d\gamma = -\frac{3}{2}\gamma \right\}, \; \left\{ \gamma \in \Omega^3_{27}M \middle| \ast d\gamma = -\frac{1}{2}\gamma \right\}, \; \left\{ \gamma \in \Omega^3_{27}M \middle| d\delta\gamma = -\frac{3}{4}\gamma \right\}.
\end{equation}
Also, the two spaces behind are included in the space $(\ref{space isomorphic to RS-fields})$, but the first space is not. Furthermore, for the third space of $(\ref{RS-field direct sum})$, there is an inclusion relation
\[
\left\{ \gamma \in \Omega^3_{27}M \middle| d\delta\gamma = -\frac{3}{4}\gamma \right\} \subset \left\{ \gamma \in \Omega^3_{27}M \middle| \Delta \gamma = -\frac{3}{4}\gamma \right\}.
\]
Since the Laplace operator $\Delta$ is a non-negative operator, we know the third space of $(\ref{RS-field direct sum})$ is zero. Thus, we obtain Theorem \ref{thmB}.
\end{proof}

\begin{corollary}
The space of the Rarita-Schwinger fields is contained in the $-1/12$-eigenspace $\left\{ \gamma \in \Omega^3_{27}M \; \middle| \; \bar{\Delta}\gamma = -\frac{1}{12}\gamma \right\}$ of the $\G_2$-Laplace operator.
\end{corollary}
\begin{proof}
The difference between the standard Laplace operator $\Delta$ and the $\G_2$-Laplace operator $\bar{\Delta}$ on the space $\Omega^3_{27}M$ is known in \cite[Proposition 5.3]{AlexandrovSemmelmann}:
\[
\bar{\Delta}\gamma = \Delta\gamma - 2\ast(d\gamma)_{\wedge^4_7} + \frac{2}{3}\ast(d\gamma)_{\wedge^4_{27}} \quad {\rm for \: all} \; \gamma \in \Omega^3_{27}M.
\]
Using the fact that the space of the Rarita-Schwinger fields is isomorphic to the space $R_{\gamma} = \left\{ \gamma \in \Omega^3_{27}M \; \middle| \; \ast d\gamma = -\frac{1}{2}\gamma \right\}$, we get this corollary.
\end{proof}

\section{Examples and Applications} \label{sec: examples and applications}
\subsection{Examples}
As shown in Theorem \ref{thmB}, we revealed that the space of the Rarita-Schwinger fields is isomorphic to $\left\{ \gamma \in \Omega^3_{27}M \mid \ast d\gamma = -\frac{1}{2}\gamma \right\}$ on compact nearly parallel $\G_2$-manifolds. In this subsection, we investigate whether non-zero Rarita-Schwinger fields actually exist on specific manifolds.

All the compact simply connected homogeneous nearly parallel $\G_2$-manifolds are classified in \cite{FKMS}. These are six cases: $(S^7,g_{\rm round}) = \Spin(7) / \G_2$, $(S^7,g_{\rm squashed}) = \frac{\Sp(2) \times \Sp(1)}{\Sp(1) \times \Sp(1)}$, $\SO(5) / \SO(3)$, $M(3,2) = \frac{\SU(3) \times \SU(2)}{\U(1) \times \SU(2)}$, $Q(1,1,1) = \SU(2)^3 / \U(1)^2$, and Aloff-Wallach spaces $N(k,l) = \SU(3) / S^1_{k,l} \;(k,l \in \mathbb{Z})$, where the embedding of $S^1 = \U(1)$ in $\SU(3)$ is given by $S^1 \ni z \mapsto \diag(z^k,z^l,z^{-(k+l)}) \in \SU(3)$. The homogeneous structure on these spaces is summarized in \cite{Singhal}. It is important that all six of these homogeneous nearly parallel $\G_2$-manifolds are naturally reductive. Furthermore, the first four examples are normal (especially standard), while the latter two are not.

By tracing the argument in Section 7 of \cite{AlexandrovSemmelmann}, we obtain the proposition about eigenvalues of the $\G_2$-Laplace operator. Let $G/H$ be a 7-dimensional oriented naturally reductive homogeneous space with reductive decomposition $\mathfrak{g} = \mathfrak{h} \oplus \mathfrak{m}$ and be standard (up to a factor). In addition, we impose the same assumption as in \cite[Lemma 7.1]{AlexandrovSemmelmann} that the 3-form derived from the torsion of the canonical homogeneous connection becomes a nearly parallel $\G_2$-structure. Then the $-1/12$-eigenspace $\left\{ \gamma \in \Omega^3_{27}M \; \middle| \; \bar{\Delta}\gamma = -\frac{1}{12}\gamma \right\}$ of the $\G_2$-Laplace operator on $\Omega^3_{27}M$ is isomorphic to the direct sum of spaces $V_{\gamma} \otimes \Hom_H(V_{\gamma},\wedge^3_{27}\mathfrak{m}^{\ast})$, on which the Casimir operator $\Cas^G_{V_{\gamma}}$ acts as $\frac{1}{160}$. Since the Casimir operator is a non-positive operator, we get 

\begin{proposition}
There are no Rarita-Schwinger fields on the four normal homogeneous nearly parallel $\G_2$-manifolds above: $(S^7,g_{\rm round})$, $(S^7,g_{\rm squashed})$, $\SO(5) / \SO(3)$, and $M(3,2)$.
\end{proposition}

On nearly K\"{a}hler 6-manifolds, which are positive Einstein manifolds with Killing spinors, the dimension of the Rarita-Schwinger fields coincides with the 3rd Betti number, and there are indeed examples where Rarita-Schwinger fields exist (cf. \cite{OhnoTomihisa}). Also, let $(M^7,g)$ be a compact $\G_2$-manifold, which is a Ricci-flat manifold. Then the dimension of the Rarita-Schwinger fields is equal to $b_2(M) + b_3(M) -1$. There are many manifolds with (non-trivial) Rarita-Schwinger fields, but no examples without Rarita-Schwinger fields have been found (cf. \cite{YasushiSemmelmann}). On the other hand, on nearly parallel $\G_2$-manifolds, which are positive Einstein manifolds with Killing spinors, we can not so far construct an example with Rarita-Schwinger fields. It is a future problem to investigate whether there exists a Rarita-Schwinger field on not normal homogeneous and inhomogeneous nearly parallel $\G_2$-manifolds (listed in \cite{FKMS}, for example).

\subsection{Linear instability}
In this subsection, as one application, we state the relationship between instability and Rarita-Schwinger fields. Various notions about stability theory are summarized in \cite{WangWang}. In this paper, we deal only with instability in the following sense. A closed Einstein manifold $(M,g)$ is linearly unstable if there exists a non-trivial symmetric tensor $H$ such that $\tr H =0, \delta H=0$ and
\begin{equation*}
\langle (\Delta - 2E)H,H \rangle < 0,
\end{equation*}
where $E$ is the Einstein constant. In particular, recall that $E=6$ on a nearly parallel $\G_2$-manifold with $\scal=42$.

For the sake of arguments, we introduce the following lemma. 
\begin{lemma}[{\cite[Proposition 6.1]{AlexandrovSemmelmann}}] \label{lem: Laplacian difference between Sym_0 and Omega^3_{27}}
Let $\gamma$ be a section of $\wedge^3_{27}M$. Then we have
\[
\mathbf{i}\Delta\mathbf{i}^{-1}(\gamma) = \Delta\gamma - 2\ast(d\gamma)_{\wedge^4_7} + 2\ast(d\gamma)_{\wedge^4_{27}} + 4\gamma.
\]
\end{lemma}

Now, we consider the space $R_{\gamma} = \left\{ \gamma \in \Omega^3_{27}M \mid \ast d\gamma = -\frac{1}{2}\gamma \right\}$ corresponding to Rarita-Schwinger fields. Through the isomorphism $\mathbf{i}^{-1}: \wedge^3_{27}M \to \Sym_0M$, the space $R_{\gamma}$ is isomorphic to the space
\[
R_H \coloneqq \left\{ H \in \Gamma(\Sym_0M) \mid 3A_{e_i \star}\bar{\nabla}_{e_i}H = -H, \delta H=0 \right\}.
\]
As already described in Theorem \ref{thmB}, the space $R_{\gamma}$ is contained in $\left\{ \gamma \in \Omega^3_{27}M \mid \Delta\gamma = \frac{1}{4}\gamma \right\}$. Employing Lemma \ref{lem: Laplacian difference between Sym_0 and Omega^3_{27}}, we also see that the space $R_H$ is contained in \\ $\left\{ H \in \Gamma(\Sym_0M) \mid \Delta H = \frac{13}{4}H \right\}$. Thus, we know that $g$ is linearly unstable on complete nearly parallel $\G_2$-manifolds with non-zero Rarita-Schwinger fields.

\section*{Acknowledgement}
We are immensely grateful to our supervisor, Professor Yasushi Homma, for innumerable support, advice, and providing us with the topic. We are also deeply grateful to Takuma Tomihisa for engaging us in various discussions.

\end{document}